\documentclass[12pt]{amsart}

\usepackage{pb-diagram}
\usepackage[all]{xy} 
\usepackage{amsmath,amsthm,amsfonts,latexsym,amscd,amssymb,enumerate}
\usepackage{graphicx}
\usepackage{color}
\usepackage[pdftex]{hyperref}
\newcommand{\Morse}{{\mathsf M}{\mathsf o}{\mathsf r}{\mathsf s}{\mathsf e}}
\swapnumbers
\theoremstyle{plain}
\newtheorem{theorem}{Theorem}[section]
\newtheorem{lemma}[theorem]{Lemma}

\newtheorem{proposition}[theorem]{Proposition}

\theoremstyle{definition}
\newtheorem{definition}[theorem]{Definition}

\theoremstyle{remark}
\newtheorem{remark}[theorem]{Remark}

\newcommand{\m}{\mathfrak{m}}

\newcommand{\reals}{\mathbb{R}}

\newcommand{\Z}{\mathbb{Z}}

\newcommand{\rationals}{\mathbb{Q}}

\newcommand{\adm}{{\mathrm{a}}{\mathrm{d}}{\mathrm{m}}}

\newcommand{\abs}[1]{\left\lvert#1\right\rvert} 

\DeclareMathOperator{\SO}{SO}

\DeclareMathOperator{\Aut}{Aut}
\DeclareMathOperator{\Out}{Out}
\DeclareMathOperator{\Isom}{Isom}

\DeclareMathOperator{\Torr}{Torr}
\DeclareMathOperator{\cent}{center}

\DeclareMathOperator{\tor}{tor}

\DeclareMathOperator{\ind}{ind}

\newcommand{\forget}[1]{}

{\catcode`@=11\global\let\c@equation=\c@theorem}

\newcommand{\Diff}{{\mathrm D}{\mathrm i}{\mathrm f}{\mathrm f}}
\newcommand{\Ver}{{\mathcal V}\!{\mathit e}{\mathit r}{\mathit t}}
\def\Riem{{\mathcal R}{\mathrm i}{\mathrm e}{\mathrm m}}
\def\p{\partial}
\def\Q{\rationals}
\def\R{\reals}
\def\Cr{{\mathrm Cr}}
\allowdisplaybreaks[2]

\begin{document}
\pagestyle{myheadings} 
\markboth{Boris Botvinnik, Bernhard Hanke, Thomas Schick and
Mark Walsh}{Homotopy groups of moduli spaces of psc-metrics}

\date{Last compiled \today}

\title[Homotopy groups of moduli spaces of psc-metrics]
{Homotopy groups of the moduli space of \\ metrics of positive
scalar curvature}

\author{Boris Botvinnik}
\address{Department of Mathematics \\
  University of Oregon, Eugene, OR, 97403 USA} 

\author{Bernhard Hanke}
\address{Zentrum Mathematik, Technische Universit\"at M\"unchen, Boltzmannstr.~3, 85748 Garching bei M\"unchen, Germany}

\author{Thomas
  Schick}
\address{Mathematisches Institut, Georg-August-Universit\"at G{\"o}ttingen\\
  Bunsenstr.~3, 37073 G\"ottingen, Germany}

\author{Mark Walsh}
\address{Department of Mathematics \\ 
University of Oregon, Eugene, OR, 97403
USA}

\maketitle
\vspace*{-5mm}

\begin{abstract}
  We show by explicit examples that in many degrees in a stable range the 
  homotopy groups of the moduli spaces of Riemannian metrics of positive
  scalar curvature on closed smooth manifolds  
  can be non-trivial. This is achieved by further
  developing and then applying a family version of the surgery
  construction of Gromov-Lawson to certain nonlinear smooth sphere bundles
  constructed 
  by Hatcher.

 As described, this works for all manifolds of suitable dimension and for the
quotient of the space of metrics of positive scalar curvature by the (free)
action of the subgroup of diffeomorphisms which fix a point and its tangent
space.

  We also construct special manifolds with positive scalar curvature where the quotient of the space of
metrics of positive scalar curvature by the full diffeomorphism group has
non-trivial higher homotopy groups.

\end{abstract}

\section{Introduction}
\subsection{Motivation}
Let $M$ be a closed smooth manifold. In this article we study the
topology of the space of metrics of positive scalar curvature $\Riem^+(M)$
and of corresponding moduli spaces.  We abbreviate ``metric of positive
scalar curvature'' by ``psc-metric''.

It has been known for a long time that there are quite a few
obstructions to the existence of psc-metrics.  This starts in
dimension $2$, where the Gau\ss -Bonnet theorem tells us that only the
sphere and $\reals P^2$ admit such a metric. In general the
Lichnerowicz formula in combination with the Atiyah-Singer index theorem 
implies that if  $M$ is a spin manifold and admits a psc-metric, then the
$\hat{A}$-genus of $M$ is zero. {The Gromov-Lawson-Rosenberg
  conjecture \cite{Rosenberg(1983)} was an
attempt to completely characterize those spin manifolds admitting
psc-metrics. It was later disproved in \cite{Sch}.}

In spite of the complicated picture for general manifolds, the existence question has been resolved
completely for simply connected manifolds $M$ of dimension at least
five.  Gromov and Lawson proved in \cite{GL1} that if $M$ is not 
spin, then there is no obstruction and $M$ admits a psc-metric. Assuming that $M$ is spin, 
{Stolz \cite{Stolz} proved that the only obstruction is the  
$KO$-valued index of the Dirac operator on $M$}.

If $M$ admits a psc-metric, one can go on and investigate the topology of 
$\Riem^+(M)$, the space of psc-metrics on $M$ equipped with the 
the $C^\infty$-topology. Note that $\Diff(M)$, the diffeomorphism group of $M$, 
acts on $\Riem^+(M)$ via pull-back, and so it is even more natural to
study the moduli space $\Riem^+(M)/\Diff(M)$.

In the spin case index theoretic methods were used  to 
show that the spaces $ \Riem^+(M)$ and
$\Riem^+(M)/\Diff(M) $ have infinitely many components in many cases, 
see e.g. the work of  Gromov-Lawson \cite{GL3} or  Lawson-Michelsohn \cite{LM} or, for more
refined versions, the papers \cite{Botvinnik-Gilkey, Leichtnam-Piazza, Piazza-Schick}. If $M$ is
simply connected, this applies to the case when $\dim(M)\equiv 1\pmod
4$.

Hitchin observed in his thesis \cite[Theorem 4.7]{Hitchin} that
sometimes, {in the spin case},  non-zero elements in the homotopy groups of $\Diff(M)$ yield, via the
action of $\Diff(M)$ on $\Riem^+(M)$, non-zero elements in the homotopy groups of 
$\Riem^+(M)$. More precisely, he proves this way that
$\pi_0(\Riem^+(M^n))$ is non-trivial for $n\equiv -1,0,1\pmod 8$ and
$\pi_1(\Riem^+(M^n))$ is non-trivial for $n\equiv -1,0\pmod 8$.

Contrasting these positive results, it has been an open problem to decide whether  
$\pi_k(\Riem^+(M))$ for $k > 1$ or $\pi_k(\Riem^+(M)/ \Diff(M))$ for $k > 0$ can be non-trivial.  
Note that, by construction, Hitchin's elements in $\pi_k(\Riem^+(S^n))$, $k=0,1$,
are mapped to zero in the moduli space $\Riem^+(M)/\Diff(M)$. 
Some experts even raised the suspicion  that the components of this moduli space are always contractible.

\subsection{Moduli spaces of psc-metrics}

In this paper we will construct many examples of non-zero elements
in higher homotopy groups of moduli spaces of psc-metrics on closed smooth manifolds $M$.  We
denote by $\Riem(M)$ the space of all Riemannian metrics with the
$C^{\infty}$-topology. The group of diffeomorphisms $\Diff(M)$ 
acts from the right on the space $\Riem(M)$ by pull-back: $(g,\phi) \mapsto \phi^*(g)$. The orbit 
space of this action is the {\em moduli space of Riemannian metrics} on $M$ and written ${\mathcal M}(M)$. The orbit 
space ${\mathcal M}^+(M)$ of the restricted $\Diff(M)$-action on  the subspace $\Riem^+(M)$ of psc-metrics, 
the {\em moduli space of Riemannian metrics of positive scalar curvature} on $M$, is  our principal object of interest.  

In general the action of the full 
diffeomorphism group is not free on $\Riem(M)$: For example, if a finite group $G$ acts effectively on $M$ (i.e. 
if $G$ occurs as a finite subgroup of $\Diff(M)$), then any metric on $M$ can 
be averaged over $G$, and the resulting metric will be fixed by $G$. 
Therefore we also consider the \emph{moduli spaces with observer} as 
proposed by Akutagawa and Botvinnik \cite{Botvinnik-Akutagawa-1}.

\begin{definition}\label{def:observerDiff}
  Let $(M,x_0)$ be a connected closed smooth manifold with some
  basepoint $x_0$. Let $\Diff_{x_0}(M)$ be the subgroup of 
  $\Diff(M)$ of those diffeomorphisms which fix $x_{0}$ and induce the
  identity on the tangent space $T_{x_0}M$. This is  the group of diffeomorphisms
which \emph{preserve an observer based at $x_0$}.
\end{definition}

\begin{lemma}\label{lem:free_action} 
  If $(M,x_0)$ is a connected smooth closed manifold with a basepoint
  $x_0$ then $\Diff_{x_0}(M)$ acts freely on the space $\Riem(M)$ of
  Riemannian metrics on $M$.
\end{lemma}
\begin{proof} {This lemma is well known, compare
    e.g.~\cite[Proposition IV.5]{Bourg}. For convenience we 
    recall the proof.}
  Assume $g$ is a Riemannian metric on $M$, $\phi\in \Diff_{x_0}(M)$
  and $\phi^*g=g$. This means that the map $\phi$ is an isometry of
  $(M,g)$. As $x_0$ and $T_{x_0}M$ are fixed by $\phi$, so are all
  geodesics emanating from $x_0$ (pointwise). Since $M$ is closed and
  connected, every point lies on such a geodesic, so $\phi$ is the
  identity.
\end{proof}

In the following we equip $\Diff(M)$ and $\Diff_{x_0}(M)$ with the $C^{\infty}$-topologies. 
Let ${\mathcal M}_{x_0}(M)=\Riem(M)/\Diff_{x_0}(M)$. We call 
${\mathcal M}_{x_0}(M)$ the \emph{observer moduli space of
  Riemannian metrics on $M$.} Since the space $\Riem(M)$ is
contractible and the action of $\Diff_{x_0}(M)$ on $\Riem(M)$ is
\emph{proper} (see \cite{Ebin}), Lemma \ref{lem:free_action} implies that
the orbit space ${\mathcal M}_{x_0}(M)$ is homotopy equivalent to the
classifying space $B\Diff_{x_0}(M)$ of the group $\Diff_{x_0}(M)$. 
In particular one obtains a $\Diff_{x_0}(M)$-principal fiber bundle
\begin{equation}\label{eq:moduli_space}
    \Diff_{x_0}(M)\to \Riem(M) \to {\mathcal M}_{x_0}(M). 
\end{equation}
This yields isomorphisms of homotopy groups
$$
\pi_{q}{\mathcal M}_{x_0}(M)= \pi_{q} B\Diff_{x_0}(M)\cong 
\pi_{q-1}\Diff_{x_0}(M), \ \ q\geq 1.
$$
Now we restrict the action of $\Diff_{x_0}(M)$ to the subspace 
$\Riem^+(M)$ of psc-metrics. Clearly this
action is free as well. We call the orbit space 
$$
{\mathcal M}^+_{x_0}(M):= \Riem^+(M)/\Diff_{x_0}(M)
$$ 
the \emph{observer moduli space of psc-metrics}.  Again we obtain a
$\Diff_{x_0}(M)$-principal fiber bundle
\begin{equation}\label{eq:Hitchins_sequence}
    \Diff_{x_0}(M)\to \Riem^+(M) \to  {\mathcal M}^+_{x_0}(M) \ .
\end{equation}
The inclusion $\Riem^+(M)\hookrightarrow \Riem(M)$
induces inclusions of moduli spaces ${\mathcal M} ^+(M) \hookrightarrow  {\mathcal M}(M)$ and 
${\mathcal M}^+_{x_0}(M) \hookrightarrow  {\mathcal M}_{x_0}(M)$. We collect our 
observations in the following lemma.

\begin{lemma} Let $M$ be a connected closed manifold and $x_0\in M$. Then
\begin{enumerate}
\item
 there is the  following commutative diagram of principal
$\Diff_{x_0}(M)$-fibrations 
\begin{equation}\label{eq:conf1}
\begin{diagram}
\setlength{\dgARROWLENGTH}{1.2em}
\node{\Riem^+(M)}
\arrow{s}
\arrow[2]{e,t}{}
\node[2]{\Riem(M)}
\arrow{s}
\\
\node{{\mathcal M}_{x_0}^+(M)}
\arrow[2]{e,t}{}
\node[2]{{\mathcal M}_{x_0}(M)}
\end{diagram}
\end{equation}
\item
the observer moduli space ${{\mathcal M}_{x_0}(M)}$
  of Riemannian metrics on $M$ is
homotopy equivalent to the classifying space $B\Diff_{x_0}(M)$;
\item
 there is a homotopy fibration
\begin{equation}\label{eq:homotopy_sequence}
    \Riem^+(M)\to {\mathcal M}_{x_0}^+(M)\to {\mathcal M}_{x_0}(M) .
\end{equation}
\end{enumerate}
\end{lemma}
The constructions of Hitchin \cite{Hitchin} use certain non-zero
elements in $\pi_k\Diff(M)$ and push them forward to the space
$\Riem^+(M)$ via the first map in \eqref{eq:Hitchins_sequence}. It is then shown that 
these elements are non-zero in  $\pi_k\Riem^+(M)$ (for $k=0,1$). 

Our main method will be similar, but starting from the fiber
sequence \eqref{eq:homotopy_sequence}. We will show that certain
non-zero elements of $\pi_k B\Diff_{x_0}(M)=\pi_k{\mathcal
  M}_{x_0}(M)$ can be lifted to ${\mathcal M}^+_{x_0}(M)$. Once such 
lifts have been constructed, it is immediate that they represent 
non-zero elements in $\pi_k{\mathcal M}^+_{x_0}(M)$ as their images are
non-zero in $\pi_k {\mathcal M}_{x_0}(M)$. 

\subsection{The results}
We start from the particular manifold $M=S^n$. Let $x_0\in S^n$ be a base point. Then
the group $\Diff_{x_0}(S^n)$ is homotopy equivalent to the group
$\Diff(D^n, \p D^n)$ of diffeomorphisms of the disk $D^n$ which
restrict to the identity on the boundary $\p D^n$. These groups and
their classifying spaces have been studied extensively.  In particular the
rational homotopy groups $\pi_q B\Diff_{x_0}(S^n)\otimes \Q$ are known
from algebraic $K$-theory computations and Waldhausen $K$-theory in a stable range.

\begin{theorem}\label{thm:Farrell-Hsiang}
{\rm (Farrell and Hsiang \cite{FH})} Let $0<k\ll n$. Then
$$
\pi_k B\Diff_{x_0}(S^n)\otimes \Q =
\left\{
\begin{array}{cl}
\Q  & \mbox{if} \ n \ \mbox{odd}, \ k=4q,
\\
0 & \mbox{else}.
\end{array}
\right.
$$
\end{theorem}
Here and in later places the shorthand notation $k \ll n$ means that for fixed $k$ there is an $N \in \mathbb{N}$ so 
that the statement is true for all $n \geq N$. 

Consider the inclusion map
$\iota: {\mathcal M}_{x_0}^+(S^n)\to
{\mathcal M}_{x_0}(S^n) = B\Diff_{x_0}(S^n)
$ 
and the corresponding homomorphism of homotopy groups:
$$ 
\iota_*: \pi_k {\mathcal M}_{x_0}^+(S^n) \to
\pi_k{\mathcal M}_{x_0}(S^n).
$$
Here is our first main result.
\begin{theorem}\label{thm:posit}
The homomorphism
$$
\iota_*\otimes\Q: \pi_k{\mathcal M}_{x_0}^+(S^n)\otimes \Q\to
\pi_k{\mathcal M}_{x_0}(S^n) \otimes \Q
$$
is an epimorphism for $0< k \ll  n $. In particular, the
groups 
$\pi_k{\mathcal M}_{x_0}^+(S^n)$ are non-trivial for odd $n$ and
$0<k=4q\ll n$. 
\end{theorem}
Theorem \ref{thm:Farrell-Hsiang} is essentially an existence theorem
and does not directly lead to a geometric interpretation of the
generators of $\pi_{k}B\Diff_{x_0}(S^n) \otimes \Q$.  This was
achieved later in the work of B\"okstedt \cite{Bok} and Igusa
\cite{Igusa-book, Igusa-new} based on a construction of certain smooth
nonlinear disk and sphere bundles over $S^{k}$ due to Hatcher.  The
non-triviality of some of these bundles is detected by the
non-vanishing of a higher Franz-Reidemeister torsion invariant.

Recall from \cite{Igusa-book, Igusa, Igusa-new} that for any closed
smooth manifold $M$ there are universal higher Franz-Reidemeister
torsion classes $\tau_{2q} \in H^{4q}(B\Torr(M); \Q)$, where $\Torr(M)
\subset \Diff(M)$ is the subgroup of diffeomorphisms of $M$ that act
trivially on $H_*(M; \Q)$. Note that $\Diff_{x_0}(S^n)\subset
\Torr(S^n)$. {Furthermore, it is well-known that
  $\Torr(S^n)$ is the subgroup of $\Diff(S^n)$} consisting of
orientation preserving diffeomorphisms.  In particular, these classes
define characteristic classes for smooth fiber bundles $M \to E \to B$
over path connected closed smooth manifolds $B$ with $\pi_1(B)$ acting
trivially on $H_*(M;\Q)$. (The last condition can be weakened to
$H_*(M;\Q)$ {being} a unipotent $\pi_1(M)$-module \cite{Igusa-new},
but this is not needed here).

The relevant class $\tau_{2q} \in H^{4q}(S^{4q};\Q)$ of the Hatcher
bundles over $S^{4q}$ with fiber $S^n$ was computed in \cite{Goette,
  Igusa-book, Igusa-new} and shown to be non-zero, if $n$ is odd.  The
generators of $\pi_k B \Diff_{x_0}(S^n)$ appearing in Theorem
\ref{thm:Farrell-Hsiang} can be represented by classifying maps $S^k
\to B\Diff_{x_0}(S^n)$ of these Hatcher bundles in this way.  In order
to prove Theorem \ref{thm:posit} we construct families of psc-metrics
on these bundles.

Therefore, in Section \ref{sec:surgery_in_families}, we will first
study how and under which conditions such constructions can be carried
out. Assuming that a given smooth bundle admits a fiberwise Morse
function, we use the surgery technique developed by Walsh
\cite{Walsh}, which generalizes the Gromov-Lawson construction of
psc-metrics via handle decompositions \cite{Ga, GL1} to families of
Morse functions, in order to construct families of psc-metrics on this
bundle, see Theorem \ref{theo:construct_metrics-1a}. This is the
technical heart of the paper at hand.  Compared to \cite{Walsh} the
novel point is the generalization of the relevant steps of this
construction to nontrivial fiber bundles.

Then, we will study particular generators of $\pi_k
B\Diff_{x_0}(S^n)\otimes \rationals$ for suitable $k$ and $n$, as in
Theorem \ref{thm:Farrell-Hsiang}.  To give a better idea how we are
going to proceed, recall that the observer moduli space ${\mathcal
  M}_{x_0}(S^n) = B\Diff_{x_0}(S^n)$ serves as a classifying space of
smooth fiber bundles with fiber $S^n$ and structure group
$\Diff_{x_0}(S^n)$.  We obtain the universal smooth fiber bundle
\[
       S^n \to   \Riem(S^n)\times_{\Diff_{x_0}(S^n)} S^n \to \Riem(S^n)/ \Diff_{x_0}(S^n) \, . 
\]
In particular, a map $f:S^k \to B\Diff_{x_0}(S^n)$ representing an element $\alpha \in \pi_k
B\Diff_{x_0}(S^n)$ gives rise to a commutative diagram of smooth fiber
bundles
\begin{equation}\label{eq:smooth-bundle}
\begin{diagram}
\setlength{\dgARROWLENGTH}{1.2em}
\node{E}
\arrow{s}
\arrow[2]{e,t}{}
\node[2]{   \Riem(S^n)\times_{\Diff_{x_0}(S^n)} S^n}
\arrow{s}
\\
\node{S^k}
\arrow[2]{e,t}{f}
\node[2]{B\Diff_{x_0}(S^n)}
\end{diagram}
\end{equation}
This shows that a lift of the class $\alpha\in \pi_k B\Diff_{x_0}(S^n)$
to $\pi_k{\mathcal M}^+_{x_0}(S^n)$ is nothing but a family of
psc-metrics of positive scalar curvature on the bundle
$E\to S^k$ from (\ref{eq:smooth-bundle}). 

We will explain the precise relationship in Section
\ref{sec:Hatchers_examples} and show that the construction described
in Section \ref{sec:surgery_in_families} applies to Hatcher's
$S^n$-bundles. Here we make use of a family of Morse functions on
these bundles as described by Goette \cite[Section 5.b]{Goette}. This
will finish the proof of Theorem \ref{thm:posit}.

Given a closed smooth manifold $M$ of dimension $n$, we can take the
fiberwise connected sum of the trivial bundle $S^k \times M \to S^k$
and Hatcher's exotic $S^n$-bundle.  Using additivity of higher torsion
invariants \cite[Section 3]{Igusa-new} we obtain non-trivial elements
in $\pi_k \mathcal{M}_{x_0}(M)$ for given $k$ for any manifold $M$ of
odd dimension $n$ as long as $ k \ll n$.

If in addition $M$ admits a psc-metric, this can be combined with the
fiberwise psc-metric on Hatcher's $S^n$-bundle constructed earlier to
obtain a fiberwise psc-metric on the resulting nontrivial $M$-bundle
over $S^k$. This shows:
\begin{theorem}\label{thm:posit_general_manifold}
Let $M$ be a closed smooth 
manifold admitting a metric $h$ of
positive scalar curvature. If $\dim M$ is odd, then the homotopy groups
$\pi_k({\mathcal M}^+_{x_0}(M),[h])$ are non-trivial for $0<k=4q\ll \dim M$.
\end{theorem}
In order to study the homotopy type of the classical moduli space of
psc-metrics it remains to construct examples of manifolds $M$ for
which the non-zero elements in $\pi_k {\mathcal M}_{x_0}^+(M)$
constructed in Theorem \ref{thm:posit_general_manifold} is not mapped
to zero under the canonical map $\pi_k {\mathcal M}^+_{x_0}(M) \to
\pi_k {\mathcal M}^+(M)$. This will be done in Section
\ref{sec:asymmetric_manifold} and leads to a proof of the following
conclusive result.
\begin{theorem}\label{thm:usual_moduli} For any $d > 0$ there exists a 
  closed smooth manifold $M$ admitting a metric $h$ of positive scalar
  curvature so that $\pi_{4q}({\mathcal M}^+(M), [h])$ is non-trivial
  for $0 < q \leq d$.
\end{theorem} 

  \begin{remark}
    One should mention that the manifolds we construct in Theorem
    \ref{thm:usual_moduli} do not admit a spin structure and are of odd
    dimension. In particular, the usual methods to distinguish elements of
    $\pi_0\mathcal{M}^+(M)$, which use the index of the Dirac operator, do not
    apply to these manifolds, and we have no non-trivial lower bound on the
    number of components of $\mathcal{M}^+(M)$.
  \end{remark}

\begin{remark}
  Finding non-zero elements of $\pi_k\Riem^+(M)$ for $k>1$ remains an
  open problem.  It would be especially interesting to find examples
  with non-zero image in $\pi_k(\Riem^+(M)/\Diff(M))$, or at least in
  $\pi_k(\Riem^+(M)/\Diff_{x_0}(M))$.

  We expect that a solution of this problem requires a different
  method than the one employed in Sections \ref{sec:Hatchers_examples}
  and \ref{sec:asymmetric_manifold} of our paper.
\end{remark}
\subsection{Acknowledgement} {Boris Botvinnik} would
like to thank K. Igusa and D. Burghelea for inspiring discussions on
topological and analytical torsion and thank SFB-478 (Geometrische
Strukturen in der Mathematik, M\"unster, Germany) and IHES for
financial support and hospitality. {Mark Walsh also would
  like to thank SFB-478 for financial support and hospitality.}
{Thomas Schick} was partially supported by the Courant
Research Center ``Higher order structures in Mathematics'' within the
German initiative of excellence.

\section{The surgery method in twisted families}
\label{sec:surgery_in_families}
The aim of this section is to prove a result on the construction of 
fiberwise metrcis of positive scalar curvature on certain smooth fiber
bundles. At first we briefly review the Gromov-Lawson
surgery technique \cite{GL1} on a single manifold. Here we use the approach
developed by Walsh \cite{Walsh,Walsh2}.
\subsection{Review of the surgery technique on a single
  manifold} \label{subsec:review_surgery} Let $W$ be a compact
manifold with non-empty boundary $\p W$ and with $\dim W = n+1$.  We
assume that the boundary $\p W $ is the disjoint union of two
manifolds $\p_0 W$ and $\p_1 W$ both of which come with collars
\begin{equation}\label{eq:colars}
\p_0 W\times [0,\epsilon)\subset W,
\ \ \ 
\p_1 W\times (1-\epsilon, 1] \subset W,
\end{equation}
where $\epsilon$ is taken with respect to some fixed reference metric
${\mathfrak m}$ on $W$, see Definition \ref{compat_metric} below.  By a
\emph{Morse function on $W$} we mean a Morse function $f: W\to [0,1]$
such that
$$
f^{-1}(0) = \p_0 W, \ \ \ f^{-1}(1) =\p_1 W
$$ 
and the restriction of $f$ to the collars (\ref{eq:colars}) coincides with
the projection onto the second factor
$$
\p_0 W\times [0,\epsilon)\to [0,\epsilon), \ \ \ 
\p_1 W\times (1-\epsilon, 1] \to (1-\epsilon, 1].
$$
We denote by $\Cr(f)$ the set of critical points of $f$.  

We say that a Morse function $f: W\to [0,1]$ is \emph{admissible} if
all its critical points have indices at most $(n-2)$ (where $\dim
W=n+1$).  We note that the last condition is equivalent to the
``codimension at least three'' requirement for the Gromov-Lawson
surgery {method}.  We denote by ${\mathsf M}{\mathsf
  o}{\mathsf r}{\mathsf s}{\mathsf e}(W)$ and $\Morse^{\adm}(W)$ the
spaces of Morse functions and admissible Morse functions,
respectively, which we equip with the $C^{\infty}$-topologies.
\begin{definition}\label{compat_metric}
  Let $f\in \Morse^{\adm}(W)$.  A Riemannian metric ${\mathfrak m}$ on
  $W$ is \emph{compatible} with the Morse function $f$ if for every
  critical point $p\in \Cr(f)$ with $\ind p=\lambda$ the positive and
  negative eigenspaces $T_pW^+$ and $T_pW^-$ of the Hessian $d^2f$ are
  $\mathfrak m$-orthogonal, and $d^2f|_{T_pW^+}=\mathfrak
  m|_{T_pW^+}$, $d^2f|_{T_pW^-}=-\mathfrak m|_{T_pW^-}$.
\end{definition}
We notice that for a given Morse function $f$, the space of
 compatible metrics is convex. Thus the space of pairs $(f,{\mathfrak
  m})$, where $f\in \Morse^{\adm}(W)$, and ${\mathfrak m}$ is a metric
compatible with $f$, is homotopy equivalent to the space
$\Morse^{\adm}(W)$. We call a pair $(f,{\mathfrak m})$ as above an
\emph{admissible Morse pair}.  We emphasize that the metric
${\mathfrak m}$ on $W$ has no relation to the psc-metrics we are going
to construct.

The ideas behind the following theorem go back to Gromov-Lawson
\cite{GL1} and Gajer \cite{Ga}. 
\begin{theorem}
  \label{GLcob} \cite[Theorem 2.5]{Walsh} Let $W$ be a smooth compact cobordism
  with $\p W= M_0\sqcup M_1$. Assume that $g_0$ is a positive scalar
  curvature metric on $M_0$ and $(f,{\mathfrak m})$ is an admissible
  Morse pair on $W$.  Then there is a psc-metric
  $\bar{g}=\bar{g}(g_0,f,{\mathfrak m})$ on $W$ which extends $g_0$
  and has a product structure near the boundary.
\end{theorem}
\begin{proof} We will provide here only an outline and refer to
  \cite[Theorem 2.5]{Walsh} for details.
\begin{figure}[!h]
\begin{picture}(0,0)%
\includegraphics{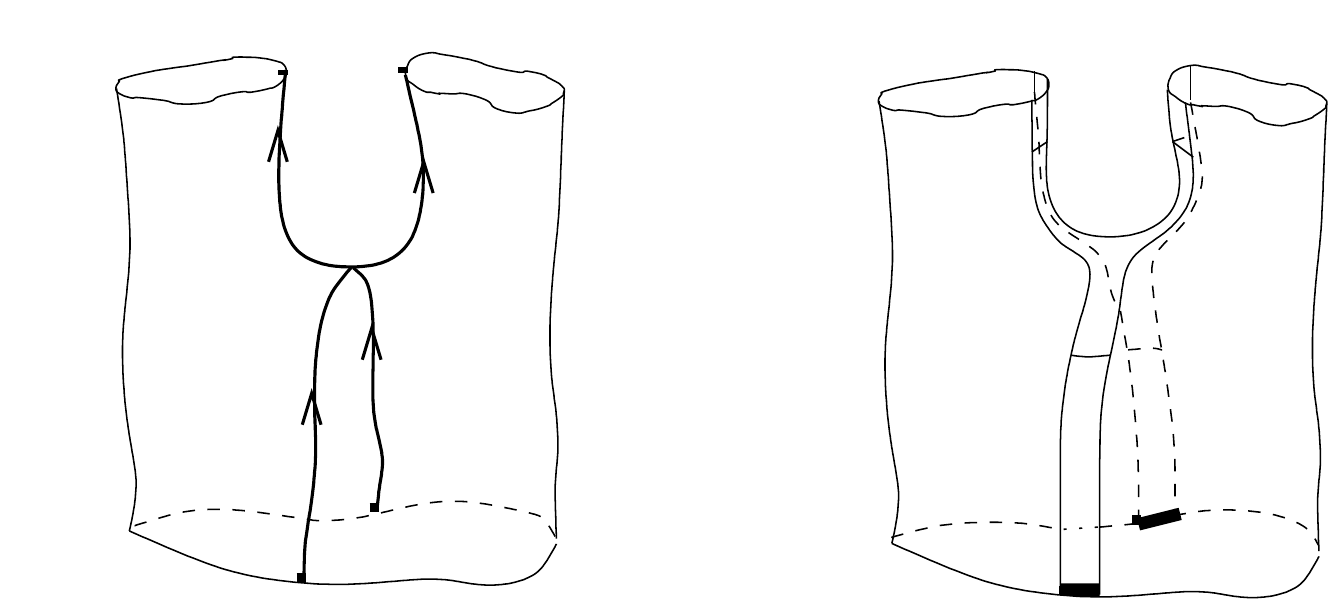} 
\end{picture}%
\setlength{\unitlength}{3947sp}%
\begingroup\makeatletter\ifx\SetFigFont\undefined%
\gdef\SetFigFont#1#2#3#4#5{%
  \reset@font\fontsize{#1}{#2pt}%
  \fontfamily{#3}\fontseries{#4}\fontshape{#5}%
  \selectfont}%
\fi\endgroup%
\begin{picture}(6381,2877)(561,-3478)
\put(2207,-1780){\makebox(0,0)[lb]{\smash{{\SetFigFont{10}{12}{\rmdefault}{\mddefault}{\updefault}{\color[rgb]{0,0,0}$w$}%
}}}}
\put(1230,-1530){\makebox(0,0)[lb]{\smash{{\SetFigFont{10}{12}{\rmdefault}{\mddefault}{\updefault}{\color[rgb]{0,0,0}$K_{+}^{q+1}(w)$}%
}}}}
\put(2445,-2568){\makebox(0,0)[lb]{\smash{{\SetFigFont{10}{12}{\rmdefault}{\mddefault}{\updefault}{\color[rgb]{0,0,0}$K_{-}^{p+1}(w)$}%
}}}}
\put(2170,-3255){\makebox(0,0)[lb]{\smash{{\SetFigFont{10}{12}{\rmdefault}{\mddefault}{\updefault}{\color[rgb]{0,0,0}$S_{-}^{p}(w)$}%
}}}}
\put(1982,-755){\makebox(0,0)[lb]{\smash{{\SetFigFont{10}{12}{\rmdefault}{\mddefault}{\updefault}{\color[rgb]{0,0,0}$S_{+}^{q}(w)$}%
}}}}
\put(5939,-3349){\makebox(0,0)[lb]{\smash{{\SetFigFont{10}{12}{\rmdefault}{\mddefault}{\updefault}{\color[rgb]{0,0,0}$N$}%
}}}}
\put(5689,-1624){\makebox(0,0)[lb]{\smash{{\SetFigFont{10}{12}{\rmdefault}{\mddefault}{\updefault}{\color[rgb]{0,0,0}$U$}%
}}}}
\put(576,-1961){\makebox(0,0)[lb]{\smash{{\SetFigFont{10}{12}{\rmdefault}{\mddefault}{\updefault}{\color[rgb]{0,0,0}$W$}%
}}}}
\put(614,-986){\makebox(0,0)[lb]{\smash{{\SetFigFont{10}{12}{\rmdefault}{\mddefault}{\updefault}{\color[rgb]{0,0,0}$M_1$}%
}}}}
\put(614,-3074){\makebox(0,0)[lb]{\smash{{\SetFigFont{10}{12}{\rmdefault}{\mddefault}{\updefault}{\color[rgb]{0,0,0}$M_0$}%
}}}}
\end{picture}%
\caption{Trajectory disks of the critical point $w$ contained inside a disk $U$}
\label{trajflow}
\end{figure}
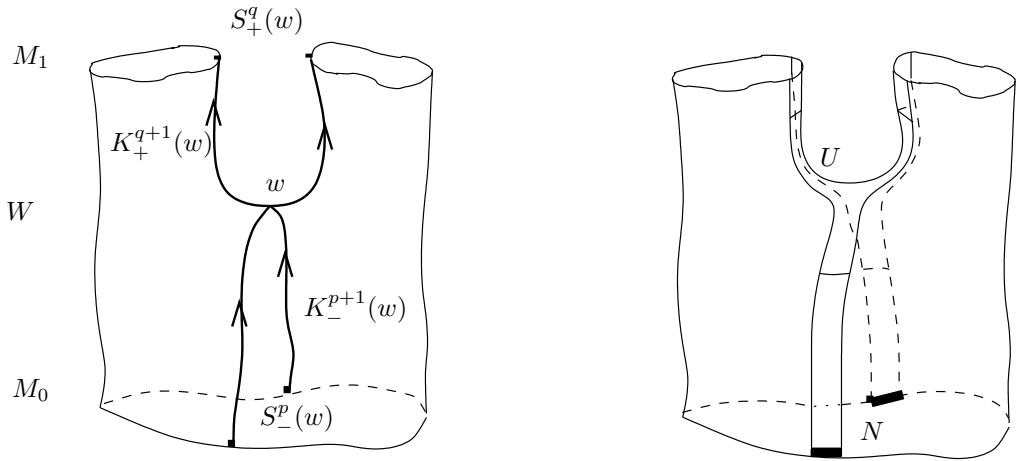
We begin with a few topological observations. For simplicity, we
assume for the moment that $W$ is an elementary cobordism, i.e. that $f$ has
a single critical point $w$ of index $p+1$. The general case is obtained by
repeating the construction for each critical point. Fix a gradient like vector
field for $f$. Intersecting transversely
at $w$ there is a pair of trajectory disks $K_{-}^{p+1}$ and $K_{+}^{q+1}$,
see Fig. \ref{trajflow}.  Here the lower $(p+1)$-dimensional disk
$K_{-}^{p+1}$ is bounded by an embedded $p$-sphere $S_{-}^{p}\subset M_0$. It
consists of the union of segments of integral curves of the gradient
vector field beginning at the bounding sphere and ending at $w$. Here
and below we use the compatible metric ${\mathfrak m}$ for all
gradient vector fields.  Similarly, $K_{+}^{q+1}$ is a
$(q+1)$-dimensional disk which is bounded by an embedded $q$-sphere
$S_{+}^{q}\subset M_1$. The spheres $S_{-}^{p}$ and $S_{+}^{q}$ are
known as trajectory spheres associated to the critical point $w$.  As
usual, the sphere $S_{-}^{p}\subset M_0$ is embedded into $M_0$
together with its neighborhood $N=S_{-}^{p}\times D^{q+1}\subset
M_0$.

We denote by $U$ the union of all trajectories originating at the 
neighborhood $N$, and notice that $U$ is a disk-shaped neighborhood
of $K_{-}^{p+1}\cup K_{+}^{q+1}$, see Fig. \ref{trajflow}. A
continuous shrinking of the radius of $N$ down to zero induces a
deformation retraction of $U$ onto $K_{-}^{p+1}\cup K_{+}^{q+1}$.

Now we consider the complement $W\setminus U$, which coincides with the
union of all trajectories originating at $M_0\setminus N$. By
assumption none of these trajectories have critical points. We use
the normalized gradient vector field of $f$ to specify a
diffeomorphism
$$
\psi : W\setminus U\to (M_0\setminus N)\times [0,1].
$$
Now we construct the metric $\bar{g}$. On the region $W\setminus U$,
we define the metric $\bar{g}$ to be simply $g_{0}|_{M_0\setminus
  N}+dt^{2}$ where the $t$ coordinate comes from the embedding $\psi$
above.
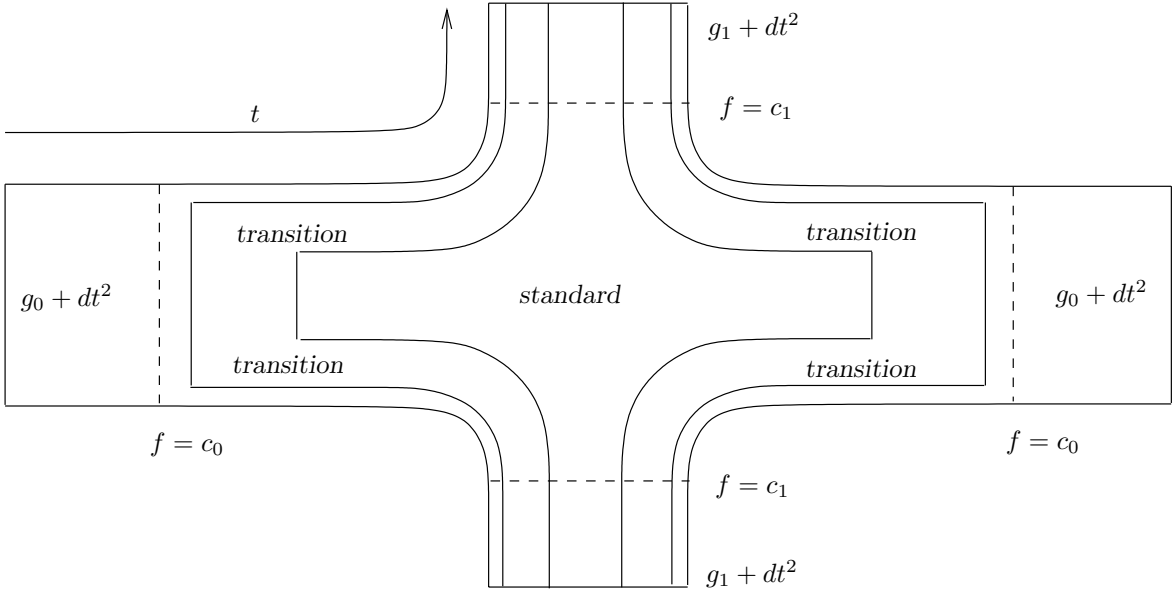
\begin{figure}[h!bp]
\begin{picture}(0,0)%
\includegraphics{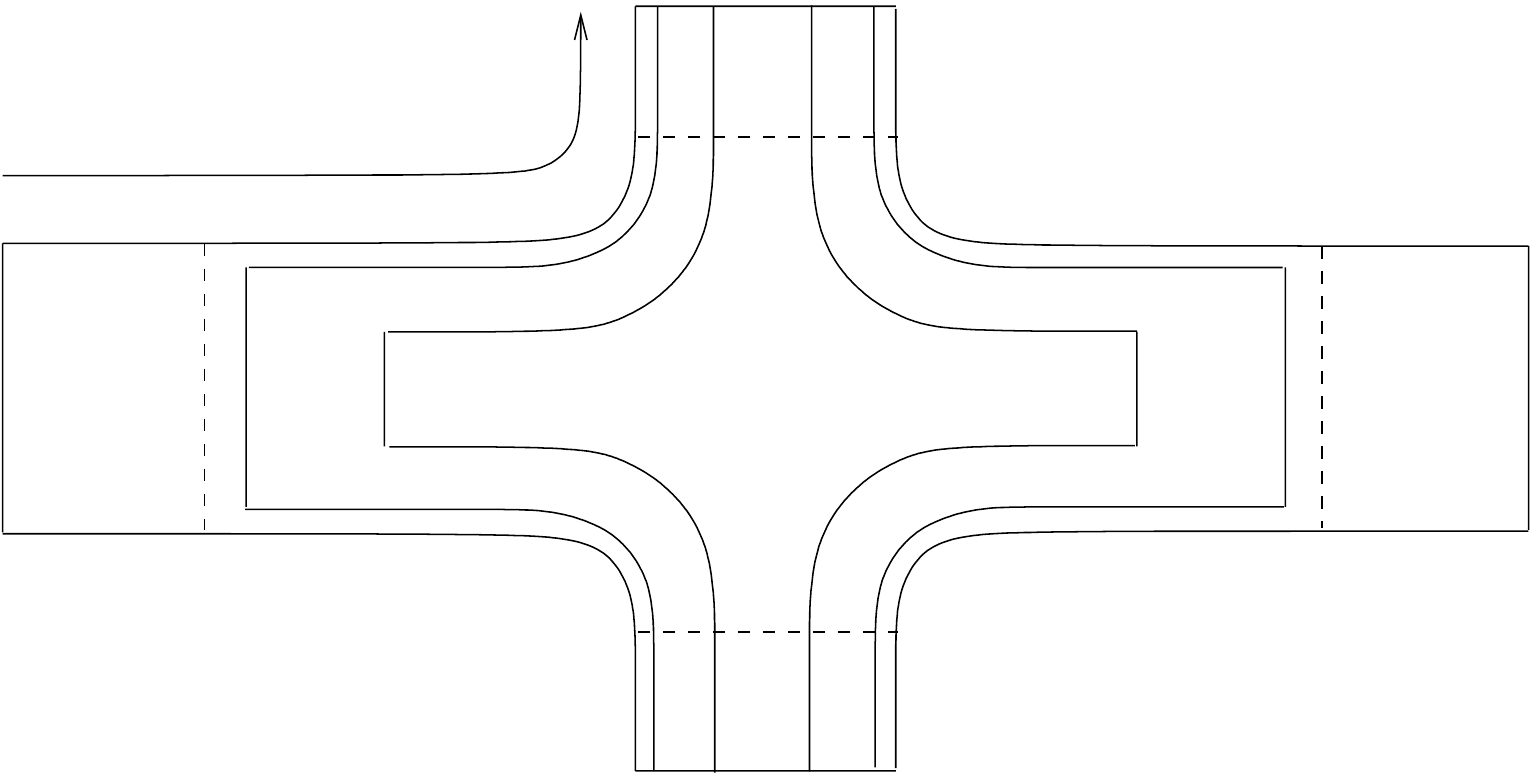}
\end{picture}%
\setlength{\unitlength}{3947sp}%
\begingroup\makeatletter\ifx\SetFigFont\undefined%
\gdef\SetFigFont#1#2#3#4#5{%
  \reset@font\fontsize{#1}{#2pt}%
  \fontfamily{#3}\fontseries{#4}\fontshape{#5}%
  \selectfont}%
\fi\endgroup%
\begin{picture}(7349,3720)(995,-3850)
  \put(4219,-2034){\makebox(0,0)[lb]{\smash{{\SetFigFont{10}{14.4}{\rmdefault}{\mddefault}{\updefault}{\color[rgb]{0,0,0}$\mbox{{\sl standard}}$}%
        }}}}
  \put(2532,-896){\makebox(0,0)[lb]{\smash{{\SetFigFont{10}{14.4}{\rmdefault}{\mddefault}{\updefault}{\color[rgb]{0,0,0}$t$}%
        }}}}
  \put(5394,-3784){\makebox(0,0)[lb]{\smash{{\SetFigFont{10}{14.4}{\rmdefault}{\mddefault}{\updefault}{\color[rgb]{0,0,0}$g_1+dt^{2}$}%
        }}}}
  \put(5407,-346){\makebox(0,0)[lb]{\smash{{\SetFigFont{10}{12}{\rmdefault}{\mddefault}{\updefault}{\color[rgb]{0,0,0}$g_1+dt^{2}$}%
        }}}}
  \put(2439,-1649){\makebox(0,0)[lb]{\smash{{\SetFigFont{10}{12}{\rmdefault}{\mddefault}{\updefault}{\color[rgb]{0,0,0}$\mbox{{\sl transition}}$}%
        }}}}
  \put(2414,-2474){\makebox(0,0)[lb]{\smash{{\SetFigFont{10}{12}{\rmdefault}{\mddefault}{\updefault}{\color[rgb]{0,0,0}$\mbox{{\sl transition}}$}%
        }}}}
  \put(6014,-1636){\makebox(0,0)[lb]{\smash{{\SetFigFont{10}{12}{\rmdefault}{\mddefault}{\updefault}{\color[rgb]{0,0,0}$\mbox{{\sl transition}}$}%
        }}}}
  \put(6014,-2499){\makebox(0,0)[lb]{\smash{{\SetFigFont{10}{12}{\rmdefault}{\mddefault}{\updefault}{\color[rgb]{0,0,0}$\mbox{{\sl transition}}$}%
        }}}}
  \put(1089,-2049){\makebox(0,0)[lb]{\smash{{\SetFigFont{10}{12}{\rmdefault}{\mddefault}{\updefault}{\color[rgb]{0,0,0}$g_0+dt^{2}$}%
        }}}}
  \put(7589,-2024){\makebox(0,0)[lb]{\smash{{\SetFigFont{10}{12}{\rmdefault}{\mddefault}{\updefault}{\color[rgb]{0,0,0}$g_0+dt^{2}$}%
        }}}}
  \put(5451,-3224){\makebox(0,0)[lb]{\smash{{\SetFigFont{10}{12}{\rmdefault}{\mddefault}{\updefault}{\color[rgb]{0,0,0}$f=c_1$}%
        }}}}
  \put(5476,-849){\makebox(0,0)[lb]{\smash{{\SetFigFont{10}{12}{\rmdefault}{\mddefault}{\updefault}{\color[rgb]{0,0,0}$f=c_1$}%
        }}}}
  \put(1901,-2961){\makebox(0,0)[lb]{\smash{{\SetFigFont{10}{12}{\rmdefault}{\mddefault}{\updefault}{\color[rgb]{0,0,0}$f=c_0$}%
        }}}}
  \put(7276,-2949){\makebox(0,0)[lb]{\smash{{\SetFigFont{10}{12}{\rmdefault}{\mddefault}{\updefault}{\color[rgb]{0,0,0}$f=c_0$}%
        }}}}
\end{picture}%
\caption{The metric $\bar{g}$ on the disk $U$}
\label{morsecoordblock}
\end{figure}
To extend this metric over the
region $U$, we have to do more work.
Notice that the boundary of $U$ decomposes as
\begin{equation*}
  \p U = (S^{p}\times D^{q+1})\cup 
(S^{p}\times S^{q}\times I)\cup (D^{p+1}\times S^{q}).
\end{equation*}
Here $S^{p}\times D^{q+1}\subset M_0$ is of course the tubular
neighborhood $N$ while the $D^{p+1}\times S^{q}\subset M_1$ piece is a tubular
neighborhood of the outward trajectory sphere $S_{+}^{q}\subset M_1$.

Without loss of generality assume that
$f(w)=\frac{1}{2}$. Let $c_0$ and $c_1$ be constants satisfying
$0<c_0<\frac{1}{2}<c_{1}<1$. The level sets $f=c_0$ and $f=c_1$ divide
$U$ into three regions: 
$$
\begin{array}{rcl}
U_0& = & f^{-1}([0,c_1])\cap U,
\\
\\
U_{w}& = & f^{-1}([c_0, c_1])\cap U,  
\\
\\
U_1 &= & f^{-1}([c_1, 1])\cap U.
\end{array} 
$$
The region $U_0$ is diffeomorphic to $N\times [0,c_1]$. We use again
the flow to identify $U_0$ with $N\times [0,c_1]$ in a way compatible
with the identification of $W\setminus U$ with $M_0\setminus N\times
I$. Then, on $U_0$, we define $\bar{g}$ as the product
$g_{0}|_{N}+dt^{2}$. Moreover, we extend this metric
$g_{0}|_{N}+dt^{2}$ near the $S^{p} \times S^{q}\times I$ part of the
boundary, where again $t$ is the trajectory coordinate.

We will now define a family of particularly useful psc-metrics on the
disk $D^k$. For a detailed discussion see \cite{Walsh}.
\begin{definition}
  Let $\delta>0$ and $\rho_{\delta}$ be a smooth
  function $\rho_{\delta}: (0,\infty)\to\R$ satisfying the following
  conditions:
  \begin{enumerate}
  \item\label{item:1}
$\rho_\delta(t)=\delta\sin{(\frac{t}{\delta})}$ when $t$ is
    near $0$;
  \item\label{item:2}
$\rho_\delta(t)=\delta$ when $t\geq\delta\cdot \frac{\pi}{2}$;
  \item\label{item:3}
$\ddot{\rho_{\delta}}(t)\leq 0$.
  \end{enumerate}
  Clearly such functions $\rho_{\delta}$ exists, furthermore, the space of
  functions satisfying \ref{item:1}, \ref{item:2}, \ref{item:3} for some
  $\delta>0$ 
  is convex. Let $r$ be 
  the standard radial distance function on $\mathbb{R}^k$, and $ds_{k-1}^2$ be
  the standard metric on $S^{k-1}$ (of radius one). Then the metric $dr^2 +
  \rho_\delta(r)^2ds_{k-1}^2$ on $(0,\infty)\times S^{k-1}$ is well-defined on
  $\R^k$. By restricting this metric to $(0,b]\times S^{k-1}$, one obtains the
  metric $g_{\tor}^{k}(\delta)$ on $D^k$. This metric is defined to be a
  \emph{torpedo metric}, see Fig. \ref{torpedo}.
\end{definition}
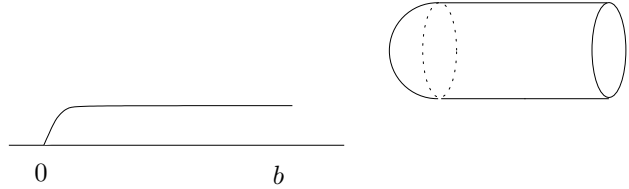
\begin{figure}[!htbp]
\begin{picture}(0,0)%
\includegraphics[height=1in]{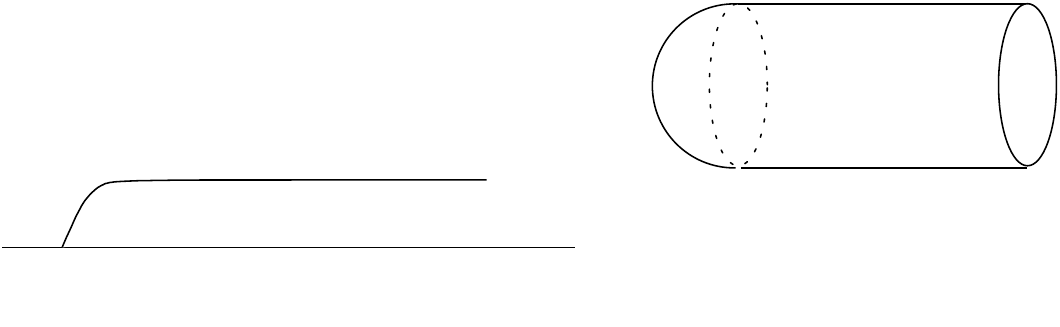}
\end{picture}%
\setlength{\unitlength}{2847sp}%
\begingroup\makeatletter\ifx\SetFigFont\undefined%
\gdef\SetFigFont#1#2#3#4#5{%
  \reset@font\fontsize{#1}{#2pt}%
  \fontfamily{#3}\fontseries{#4}\fontshape{#5}%
  \selectfont}%
\fi\endgroup%
\begin{picture}(5079,1559)(1902,-7227)
\put(2114,-7136){\makebox(0,0)[lb]{\smash{{\SetFigFont{10}{8}{\rmdefault}{\mddefault}{\updefault}{\color[rgb]{0,0,0}$0$}%
}}}}
\put(4189,-7161){\makebox(0,0)[lb]{\smash{{\SetFigFont{10}{8}{\rmdefault}{\mddefault}{\updefault}{\color[rgb]{0,0,0}$b$}%
}}}}
\end{picture}%
\caption{A torpedo function and the resulting torpedo metric}
\label{torpedo}
\end{figure} 
\begin{remark}
  It is easy to show that the above conditions \ref{item:1}, \ref{item:2},
  \ref{item:3} guarantee 
  that $g_{\tor}^{k}(\delta)$ has positive scalar curvature. Moreover
  it is $SO(k)$-symmetric and is a product with the standard metric on
  the $(k-1)$-sphere of radius $\delta$ near the boundary of $D^k$ and
  is the standard metric on the $k$-sphere of radius $\delta$ near the
  center of the disk. Also one can show that the scalar curvature of
  $g_{\tor}^{k}(\delta)$ can be bounded below by an arbitrarily large
  constant by choosing $\delta$ sufficiently small.
\end{remark}
The most delicate part of the construction, carried out carefully in
\cite{Walsh}, involves the following: Inside the region $U_w$, which is
identified with the product $D^{p+1}\times D^{q+1}$, the metric
smoothly passes  into a standard product
$g_{\tor}^{p+1}(\epsilon)+g_{\tor}^{q+1}(\delta)$ for some
appropriately chosen $\epsilon,\delta>0$, globally keeping the scalar
curvature positive. This is done so that the induced metric on the
level set $f^{-1}(c_1)$, denoted $g_1$, is precisely the metric
obtained by applying the Gromov-Lawson construction to
$g_0$. Furthermore, near $f^{-1}(c_1)$ we have $\bar{g}=g_1+dt^{2}$. Finally,
on $U_{1}$, which is identified with $D^{p+1}\times S^{q}\times
[c_1,1]$ in the usual manner, the metric $\bar{g}$ is simply the
product $g_1+dt^{2}$. See Fig.~\ref{morsecoordblock} for an
illustration.

After the choice of the Morse coordinate diffeomorphism with $D^{p+1}\times
D^{q+1}$ (and of the other parameters like $\epsilon$ and $\delta$), the
construction is explicit and depends continuously on the given metric $g_0$ on
$S^p\times D^{q+1}$. 

Later on we will need the following additional facts. The next lemma
is proved in \cite[Section 3]{Walsh}.
\begin{lemma}\label{lem:isotopy}
  The ``initial'' transition consists of an isotopy. In particular, $g_0$ is
  isotopic to a metric which, on a neighborhood diffeomorphic to $S^p\times
  D^{q+1}$ of the surgery sphere $S^p_-$ in $M_0$, is
  $\delta^2ds^2_p+g_{\tor}^{q+1}(\delta)$. 
\end{lemma}
\begin{lemma}\label{symmetry}
  The whole construction is $O(p+1)\times O(q+1)$-equivariant.
\end{lemma}
\begin{proof}
  By construction, the standard product of torpedo metrics even is
  $O(p+1)\times O(q+1)$-invariant. It is a matter of carefully going through
  the construction of the transition metric in \cite{Walsh} to check that this
  construction is equivariant for the obvious action of these groups. This is
  done in \cite[Lemma 2.2]{Walsh2}. 
\end{proof}
Lemma \ref{symmetry} will be of crucial importance later, when in a non-trivial
family we cannot choose globally defined Morse coordinates giving
diffeomorphisms to $D^{p+1}\times D^{q+1}$ (as the bundle near the critical
set is not trivial). We will construct Morse coordinates well defined up to
composition with elements of $O(p+1)\times O(q+1)$. The equivariance of Lemma
\ref{symmetry} then implies that our construction, which a priori depends on
the choice of these coordinates, is consistent and gives rise to a smooth
globally defined family of metrics.

We should emphasize that this construction can be carried out for a
tubular neighborhood $N$ of arbitrarily small radius and for $c_0$
and $c_1$ chosen arbitrarily close to $\frac{1}{2}$. Thus the region
$U_w$, on which the metric $\bar{g}$ is not simply a product and is
undergoing some kind of transition, can be made arbitrarily small with
respect to the background metric ${\mathfrak m}$. As critical points
of a Morse function are isolated, it follows that this construction
generalizes easily to Morse functions with more than one critical
point.
\end{proof}
\subsection{Extension to families}
There is a number of ways to generalize the surgery procedure to
families of manifolds. A construction relevant to our goals leads to
families of Morse functions, or maps with \emph{fold} singularities.
We start with a local description. 
\begin{definition} \label{def:standard_fold} A map $F: \R^k\times \R^{n+1}\to \R^k\times
    \R$ is called a {\em standard map with a fold singularity} of index $\lambda$, 
    if there is a $c \in \R$ so that $f$ is given as
\begin{equation}\label{fold}
\begin{array}{rcl}
\R^k\times \R^{n+1}&\longrightarrow& \R^k\times \R,
\\
(y,x) & \longmapsto & \left( y, c-x_1^2-\cdots-x_{\lambda}^2+
x_{\lambda+1}^2+\cdots+x_{n+1}^2\right).
\end{array}
\end{equation}
Roughly speaking, the composition 
\[
   \R^k \times \R^{n+1} \stackrel{F}{\to} \R^k \times \R \stackrel{p_2}{\to} \R
\]
with the projection $p_2$ onto the second factor defines a $\R^k$-parameterized family of Morse functions of index $\lambda$ 
on $\R^{n+1}$
in standard form.   

\end{definition}
Let $W$ be a compact manifold with boundary $\p W \neq \emptyset$,
$\dim W= n+1$. We denote by $\Diff(W,\p W)$ the group of all
diffeomorphisms of $W$ which restrict to the identity near the
boundary $\p W$.  Then we consider a smooth fiber bundle $\pi: E\to B$
with fiber $W$, where $\dim B= k$ and $\dim E=n+1+k$. The structure
group of this bundle is assumed to be $\Diff(W,\p W)$ and the base
space $B$ to be a compact smooth manifold.  Assume that the boundary
$\p W$ is split into a disjoint union: $\p W = \p_0 W\sqcup \p_1 W$.

Let $\pi_0: E_0\to B$, $\pi_1: E_1\to B$ be the restriction of the
fiber bundle $\pi: E\to B$ to the fibers $\p_0 W$ and $\p_1 W$
respectively. Since each element of the structure group $\Diff(W,\p W)$
restricts to 
the identity near the boundary, the fiber bundles $\pi_0: E_0\to B$,
$\pi_1: E_1\to B$ are trivialized:
$$
E_0 = B \times \p_0 W \stackrel{\pi_0}{\longrightarrow} B, \ \ \ E_1
=B\times \p_1 W \stackrel{\pi_1}{\longrightarrow} B.
$$
Choose a splitting of the tangent bundle $\tau_E$ of the total space
as $\tau_E\cong \pi^*\tau_B\oplus \Ver$, where $\Ver$ is the bundle
tangent to the fibers $W$, i.e.~choose a connection.  
\begin{definition}\label{folds}  
Let  $\pi: E\to B$ be a smooth bundle as above. For each $z$ in $B$ let 
\[
    i_z : W_z \to E 
\]
be the inclusion of the fiber $W_z := \pi^{-1}(z)$. 
Let $F: E \to B\times I$ be a smooth map. 
The map $F$ is said to be an admissible family of Morse
  functions or \emph{admissible with fold singularities 
with respect to $\pi$} if it satisfies the following conditions:
\begin{enumerate}
\item
 The diagram
$$
\begin{diagram}
\setlength{\dgARROWLENGTH}{1.6em}
 \node{}
\node[2]{E}
      \arrow{s,l}{\pi}
      \arrow[2]{e,t}{F}
\node[2]{B\times I}
      \arrow{wsw,b}{p_1}
\\
\node[3]{B}
\end{diagram}
$$
commutes.  Here $p_1: B\times I\to B$ is projection on the first
factor.
\item
The pre-images $F^{-1}(B\times\{0\})$ and  
$F^{-1}(B\times\{1\})$ coincide with the submanifolds $E_0$ and $E_1$ 
respectively. 
\item
The set $\Cr(F)\subset E$ of critical points of $F$
  is contained in $E\setminus (E_0\cup E_1)$ and near each critical
  point of $F$ the bundle $\pi$ is equivalent to the trivial bundle
  $\R^k \times \R^{n+1} \stackrel{p_1}{\to} \R^k$ so that with respect
  to these coordinates on $E$ and on $B$ the map $F$ is a standard map
  $\R^{k} \times \R^{n+1} \to \R^k \times \R$ with a fold singularity
  as in Definition \ref{def:standard_fold}
\item
For each $z\in B$ the restriction
$$
f_z=F|_{W_z}: W_z\to \{z\}\times I \stackrel{p_2}{\longrightarrow} I
$$
is an admissible Morse function as in Subsection
\ref{subsec:review_surgery}.  In particular, its critical points have
indices $\leq n-2$.
\end{enumerate}  
\end{definition}
We assume in addition that the smooth bundle $\pi: E\to B$ is a
Riemannian submersion $\pi: (E,{\mathfrak m}_{E})\to (B,{\mathfrak
  m}_{B})$, see \cite{Besse}. Here we denote by ${\mathfrak m}_{E}$
and ${\mathfrak m}_{B}$ the metrics on $E$ and $B$ corresponding to
the submersion $\pi$. Now let $F: E \to B\times I$ be an admissible
map with fold singularities with respect to $\pi$ as in Definition
\ref{folds}. If the restriction ${\mathfrak m}_{z}$ of the submersion
metric ${\mathfrak m}_{E}$ to each fiber $W_z$, $ z \in B$, is
compatible with the Morse function $f_z=F|_{W_z}$, we say that the
\emph{metric ${\mathfrak m}_{E}$ is compatible with the map $F$.}
\begin{proposition}
  Let $\pi: E\to B$ be a smooth bundle as above and $F: E \to B\times
  I$ be an admissible map with fold singularities with respect to
  $\pi$. Then the bundle $\pi: E\to B$ admits the structure of a
  Riemannian submersion $\pi: (E,{\mathfrak m}_{E})\to (B,{\mathfrak
    m}_{B})$ such that the metric ${\mathfrak m}_{E}$ is compatible
  with the map $F: E \to B\times I$.
\end{proposition}
\begin{proof}
  One can choose a Riemannian metric ${\mathfrak m}_{B}$ on the base
  $B$, and for each fiber $W_z$ there is a metric ${\mathfrak m}_{z}$
  compatible with the Morse function $f_z=F|_{W_z}$. Using convexity
  of the set of compatible metrics and the local triviality in the definition
  of a family of Morse functions, we can choose this family to depend
  continuously on $z$. Then one
  can choose an integrable distribution (sometimes called
    connection) to construct a submersion
  metric ${\mathfrak m}_{E}$ which is compatible with the map $F: E
  \to B\times I$, see \cite{Besse}.
\end{proof}
Below we assume that the fiber bundle $\pi : E\to B$ is given the
structure of a Riemannian submersion $\pi: (E,{\mathfrak m}_{E})\to
(B,{\mathfrak m}_{B})$ such that the metric ${\mathfrak m}_{E}$ is
compatible with the map $F: E \to B\times I$. 

Consider the critical set $\Cr(F)\subset E$. It follows from the
definitions that $\Cr(F)$ is a smooth $k$-dimensional submanifold in
$E$, and it splits into a disjoint union of path components
(``folds'')
$$
\Cr(F)= \Sigma_1\sqcup\cdots\sqcup\Sigma_s \, . 
$$
Furthermore,
it follows that the restriction of the fiber projection
$$
\pi|_{\Sigma_j}:  \Sigma_j\longrightarrow B
$$
is a local diffeomorphism for each $j=1,\ldots,s$. In particular,
$\pi|_{\Sigma_j}$ is a covering map, and if the base $B$
is simply-connected then $\pi|_{\Sigma_j}$ is a diffeomorphism onto its
image.

Since the metric ${\mathfrak m}_{E}$ is a submersion metric, the
structure group of the vector bundle $\Ver\to E$ is reduced to
$O(n+1)$. Furthermore, since the metrics ${\mathfrak m}_z$ are
compatible with the Morse functions $f_z=F|_{W_z}$, the restriction
$\Ver|_{\Sigma_j}$ to a fold $\Sigma_j\subset \Cr(F)$ splits further
orthogonally into the positive and negative eigenspaces of the Hessian
of $F$. Thus the metric ${\mathfrak m}_E$ induces the splitting of the
vector bundle
$$
\Ver|_{\Sigma_j}\cong \Ver_j^-\oplus \Ver_j^+
$$
with structure group $O(p+1)\times O(q+1)$ for each
$\Sigma_j$. Here is the main result of this section:
\begin{theorem}\label{theo:construct_metrics-1a}
  Let $\pi: E\to B$ be a smooth bundle, where the fiber
  $W$ is a compact manifold with boundary $ \p W= M_0\sqcup M_1$, the
  structure group is $\Diff(W,\p W)$ and the base space $B$ is a
  compact smooth simply connected manifold.  Let $F: E\to
  B\times I$ be an admissible 
  map with fold singularities with respect to $\pi$. In addition, we
  assume that the fiber bundle $\pi : E\to B$ is given the structure
  of a Riemannian submersion $\pi: (E,{\mathfrak m}_{E})\to
  (B,{\mathfrak m}_{B})$ such that the metric ${\mathfrak m}_{E}$ is
  compatible with the map $F: E \to B\times I$.  Finally, we assume
  that we are given a smooth map $g_0 : B \to \Riem^+(M_0)$.

  Then there exists a Riemannian metric
  $\bar{g}=\bar{g}(g_0,F,{\mathfrak m}_{E})$ on $E$ such that for each
  $z\in B$ the restriction $\bar{g}(z)=\bar{g}|_{W_z}$
  to the fiber $W_z=\pi^{-1}(z)$ satisfies the
  following properties:
\begin{enumerate}
\item
 $\bar{g}(z)$ extends $g_0(z)$;
\item
 $\bar{g}(z)$ is a product metric $g_{\nu}(z)+dt^2$
  near $M_{\nu}\subset \p W_z$, $\nu=0,1$;
\item
$\bar{g}(z)$ has positive scalar curvature
  on $W_z$.
\end{enumerate}
\end{theorem} 
\begin{proof} 
  We assume that $B$ is path-connected. Let $\dim B=k$, $\dim W=
  n+1$. We denote, as above, $\Cr(F)=
  \Sigma_1\sqcup\cdots\sqcup\Sigma_s, $ where the $\Sigma_j$ is a
  path-connected fold.  For a given point $z\in B$, we denote by
  $f_z=F|_{W_z} : W_z \to I$ the corresponding admissible Morse
  function.

  The metric $\bar{g}$ will be constructed by a method which is quite
  similar to that employed in the proof of Theorem \ref{GLcob}. We
  begin by equipping the boundary component $E_0$ with
  the given Riemannian metric ${g}_0$. We choose a
  gradient-like vector field $V$ and use the trajectory flow of $V$ to
  extend $\bar{g}_0$ as a product metric away from the folds
  $\Cr(F)$. Near the folds $\Cr(F)$, some modification is
  necessary. However, roughly speaking, the entire construction goes
  through in such a way that the restriction to any fiber is the
  construction of Theorem \ref{GLcob}.

  We will initially assume that $\Cr(F)$ has exactly one
  path-connected component $\Sigma$. The more general case will follow
  from this by iterated application of the construction. We will denote by $c$
  the critical value
  associated with the fold $\Sigma$, i.e. $p_2\circ F(\Sigma)=c\in
  I$. Let $\epsilon_c>0$ be small. Let $V$ denote the
  normalized gradient vector field associated to $F$ and ${\mathfrak
    m}_{E}$ which is well-defined away from the singularities of $F$.  As
  $F$ has no other critical values, we use $V$ to specify a
  diffeomorphism
\begin{equation*}
\begin{split}
\phi_0:E_0\times [0,c-\epsilon_c]
&\longrightarrow F^{-1}(B\times[0,c-\epsilon_c])\\ 
(w,t)&\longmapsto (h_w(t)),
\end{split}
\end{equation*}
where $h_w$ is the integral curve of $V$ beginning at $w$. In
particular, $p_2\circ F\circ\phi_0$ is the projection onto
$[0,c-\epsilon_c]$. As the bundle $\pi_0:E_0\rightarrow B$ is trivial,
this gives rise to a diffeomorphism 
$$
B\times M_0\times [0,c-\epsilon_c]\cong F^{-1}(B\times[0,c-\epsilon_c]).
$$ 
Let $\bar{g}_{c-\epsilon_c}$ denote the metric obtained on
$F^{-1}(B\times[0,c-\epsilon_c])$ by pulling back, via this
diffeomorphism, the warped product metric $\m_B+g_0+dt^{2}$.  In order
to extend this metric past the fold $\Sigma$, we must adapt our
construction.

Our next goal is to construct a metric $\bar{g}_{c+\epsilon_c}$ on
$F^{-1}(B\times [0,c+\epsilon_c])$, so that on each fiber
$$
\pi^{-1}(y)\cap{F^{-1}(B \times [0,c+\epsilon_c])}
$$
the induced metric has positive scalar curvature and is a product near
the boundary. Fiberwise, this is precisely the situation dealt with in
Theorem \ref{GLcob}. However, performing this over a family of Morse critical
points, we must ensure compatibility of our construction over
the entire family. The main problem is that our construction depends
on the choice of ``Morse coordinates'', i.e.~the diffeomorphism of a
neighborhood of the critical point to $D^{p+1}\times D^{q+1}$. Because
of the non-triviality of the bundle, a global choice of this kind is
in general not possible. We will normalize the situation in such a way
that we choose diffeomorphisms up to precomposition with elements of
$O(p+1)\times O(q+1)$ (in some sense a suitable reduction of the
structure group). We then use Lemma \ref{symmetry}, that the
construction employed is equivariant for this smaller group
$O(p+1)\times O(q+1)$.

Our strategy actually is to use the fiberwise exponential map for
$\mathfrak m_z$ at the critical set as Morse coordinates. Because of
the canonical splitting $\Ver|\Sigma=\Ver^-\oplus \Ver^+$ with
structure group $O(p+1)\times O(q+1)$ this gives coordinates which are
well defined up to an action of $O(p+1)\times O(q+1)$ (the choice of
orthonormal bases in $\Ver^+$ and $\Ver^-$). However, these
coordinates are \emph{not} Morse coordinates for $F$. That the metrics
$\mathfrak m_z$ are compatible with the Morse function $f_z$ only
means that this is the case infinitesimally. We will therefore deform
the given Morse function $F$ to a new Morse function $F_1$ for which
our coordinates are Morse coordinates.

We denote by $D\Ver_{\Sigma}$ the corresponding disk bundle of radius
$\delta$ with respect to the background metric $\m_E$. For each $w\in
\Sigma$, we denote by $D_w(\Ver_\Sigma)$ the fiber of this bundle. If
$\delta$ is sufficiently small, the fiberwise exponential map (and
local orthonormal bases for $\Ver^+$ and $\Ver^-$) define coordinates
$D^{p+1}\times D^{q+1}$ for neighborhoods of the critical point in
each fiber. We use the exponential map to pull back all structures to
$D^{p+1}\times D^{q+1}$ and, abusing notation, denote them in the old
way. In particular, the function $F$ is defined on $D^{p+1}\times
D^{q+1}$.

Let $\rho$ and $r$ denote the distance to the origin in $D^{p+1}$ and
$D^{q+1}$, respectively. Then $\rho^2$ and $r^2$ are smooth functions
on the image under the fiberwise exponential map of
$D(\Ver_\Sigma)$. Moreover, define $F_{std}\colon D\Ver_\Sigma\to\R$
by $F_{std}:=c-\rho^2+r^2$. The compatibility condition on $F$ and the
Taylor expansion theorem imply that $F-F_{std} =
O(\sqrt{r^2+\rho^2}^3)$, i.e.~$F-F_{std}$ is cubic in the $\mathfrak
m_z$-distance to the origin.

Choose a sufficiently small $\alpha>0$ and a smooth cutoff function
$\phi_\alpha\colon\R\to [0,1]$ with
\begin{enumerate}\item 
  $\phi_\alpha(s)=1$ for $s<\alpha$
  \item $\phi_\alpha(s)=0$ for $s>2\alpha$
  \item $\abs{\phi'(s)}\le 10/\alpha$ $\forall s\in\R$.
\end{enumerate}
Then $F_t:= F_{std} + (1-t\phi_\alpha(\sqrt{\rho^2+r^2}))(F-F_{std})$ provides
a homotopy between $F=F_0$ and $F_1$ of families of Morse functions with the
following properties: 
\begin{enumerate}
\item
$\Cr(F_t)=\Cr(F)$ $\forall t\in [0,1]$; 
\item
 $F_t$ coincides with $F$ outside of a tubular neighborhood
of  $\Sigma$ $\forall t\in [0,1]$;
\item
$F_1=F_{std}$ on a sufficiently small neighborhood of
  the fold $\Sigma$ in $D\Ver_\Sigma$.
\end{enumerate} 
\begin{figure}[htbp]
\begin{picture}(0,0)%
\includegraphics{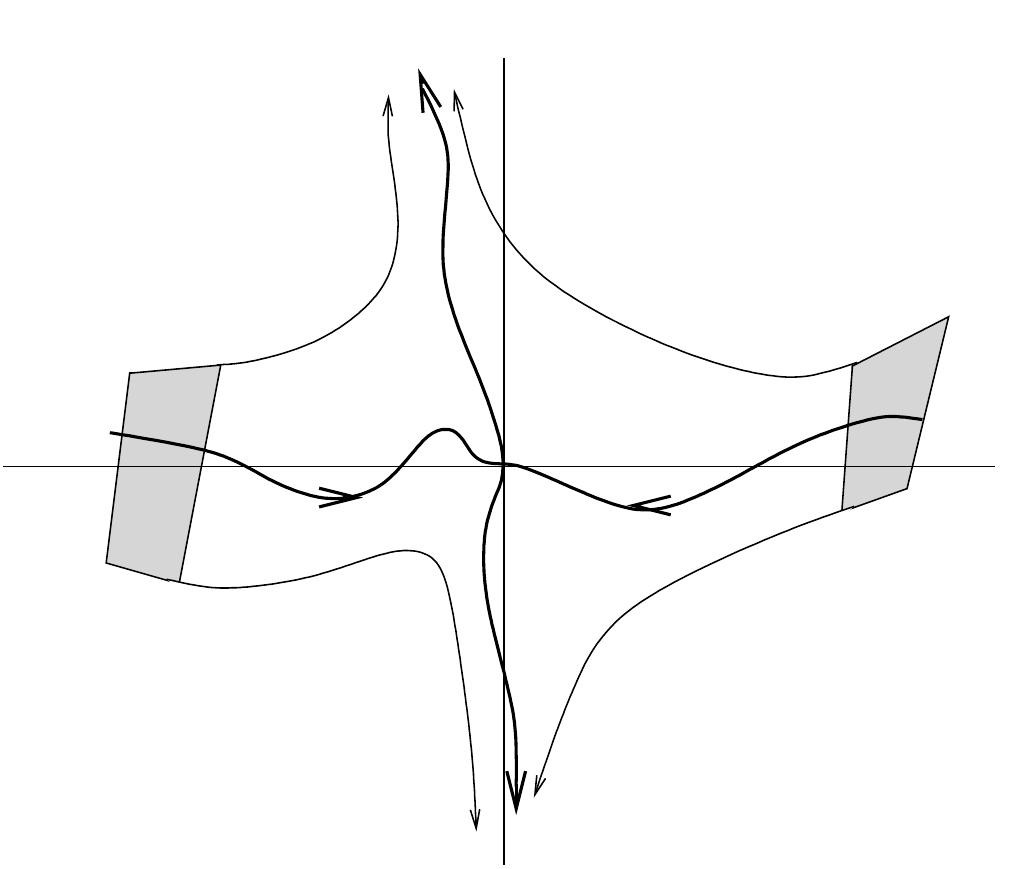}
\end{picture}%
\setlength{\unitlength}{3947sp}%
\begingroup\makeatletter\ifx\SetFigFont\undefined%
\gdef\SetFigFont#1#2#3#4#5{%
  \reset@font\fontsize{#1}{#2pt}%
  \fontfamily{#3}\fontseries{#4}\fontshape{#5}%
  \selectfont}%
\fi\endgroup%
\begin{picture}(4921,4157)(1995,-6514)
  \put(3914,-4161){\makebox(0,0)[lb]{\smash{{\SetFigFont{10}{12}{\rmdefault}{\mddefault}{\updefault}{\color[rgb]{0,0,0}$D_{w}^{q+1}$}%
        }}}}
  \put(3726,-3349){\makebox(0,0)[lb]{\smash{{\SetFigFont{10}{12}{\rmdefault}{\mddefault}{\updefault}{\color[rgb]{0,0,0}$t$}%
        }}}}
  \put(3226,-4886){\makebox(0,0)[lb]{\smash{{\SetFigFont{10}{12}{\rmdefault}{\mddefault}{\updefault}{\color[rgb]{0,0,0}$D_{w}^{p+1}$}%
        }}}}
  \put(4426,-2511){\makebox(0,0)[lb]{\smash{{\SetFigFont{10}{12}{\rmdefault}{\mddefault}{\updefault}{\color[rgb]{0,0,0}$D\Ver_{w}^{+}$}%
        }}}}
  \put(6901,-4624){\makebox(0,0)[lb]{\smash{{\SetFigFont{10}{12}{\rmdefault}{\mddefault}{\updefault}{\color[rgb]{0,0,0}$D{\Ver}_{w}^{-}$}%
        }}}}
\end{picture}%
\caption{The images of the trajectory disks $D_{w}^{p+1}$ and
  $D_{w}^{q+1}$ in $D_w\Ver(\Sigma)$ after application of the inverse
  exponential map}
\label{perturb}
\end{figure}
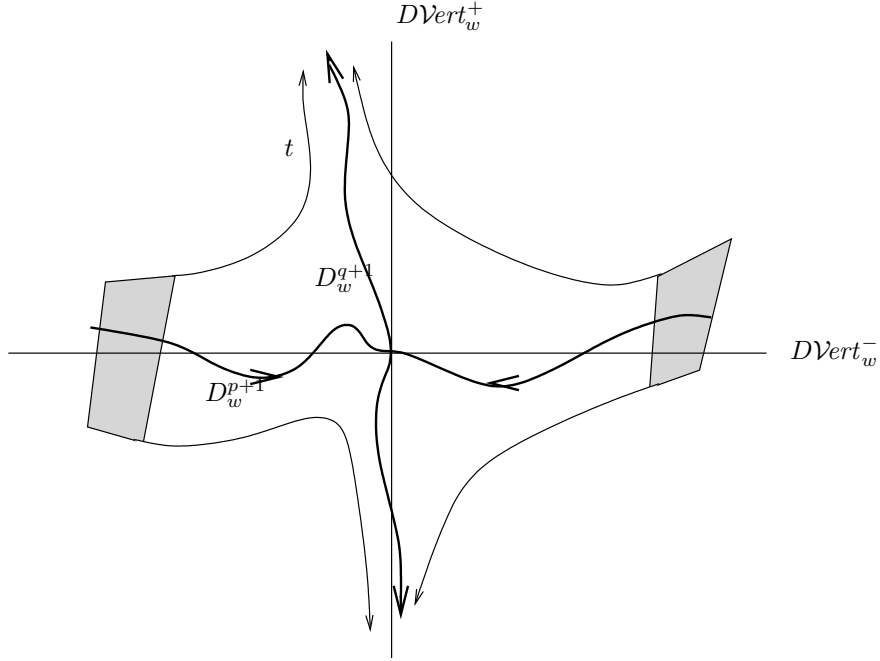 
The second and the third condition are evident. For the first, we have
to check that we did not introduce new critical points. Now the
gradient of $F_{std}$ is easily calculated and its norm at $x$ is
equal to the norm of $x$. On the other hand, the gradient of
$$(1-t\phi(\sqrt{r^2+\rho^2}))(F-F_{std})$$ 
has two summands:
\begin{enumerate}
\item $(1-t\phi(\sqrt{r^2+\rho^2}))\nabla(F-F')$ where $(1-t\phi)$ is
  bounded and $\nabla(F-F')$ is quadratic in the distance to the
  origin (as $F-F'$ has a Taylor expansion which starts with cubic
  terms).
\item $t\phi'(\sqrt{r^2+\rho^2})\nabla(\sqrt{r^2+\rho^2})(F-F')$. This
  vanishes identically if $r^2+\rho^2\le \alpha^2$, and is bounded by
  $10 (F-F')/\alpha0\le 10(F-F')/\sqrt{r^2+\rho^2} $ if $r^2+\rho^2\ge
  \alpha^2$ (here we use that the gradient of the distance to the
  origin $\sqrt{r^2+\rho^2}$ has norm $1$). Since $F-F'$ is cubic in
  $\sqrt{r^2+\rho^2}$, the whole expression is quadratic.
\end{enumerate}
It follows that, if $\alpha$ is chosen small enough (there is a
uniform bound because we deal with a compact family, so we find
uniform bounds for the implicit constants in the above estimates), the
gradient of $F_t$ vanishes exactly at the origin. Near the origin, by
the choice of $\phi$, $F_t=tF_{std}+(1-t)F$. Because the Hessians of
$F$ and of $F_{std}$ are identical at the origin, the Hessian of $F_t$
also coincides with the Hessian of $F_{std}$, in particular $F_t$ is a
family of Morse functions. To find the required local Morse
coordinates, we can invoke Igusa's \cite[Theorem 1.4]{Igusa-book}.

Thus we can assume that the map $F$ is standard near the fold
$\Sigma$, i.e. $F=F_1$ in the first place, and from now on we will do so.
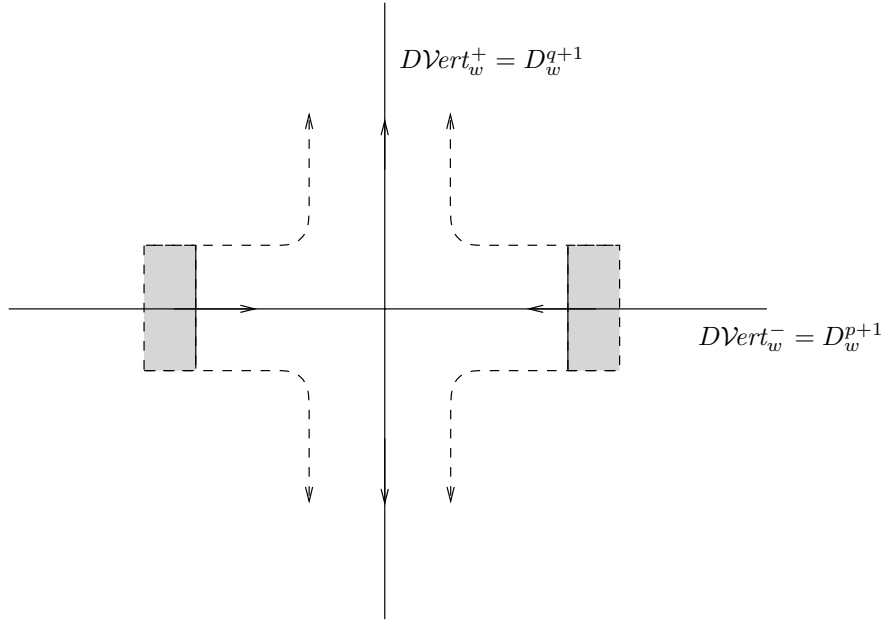
\begin{figure}[htbp]
\begin{picture}(0,0)%
\includegraphics{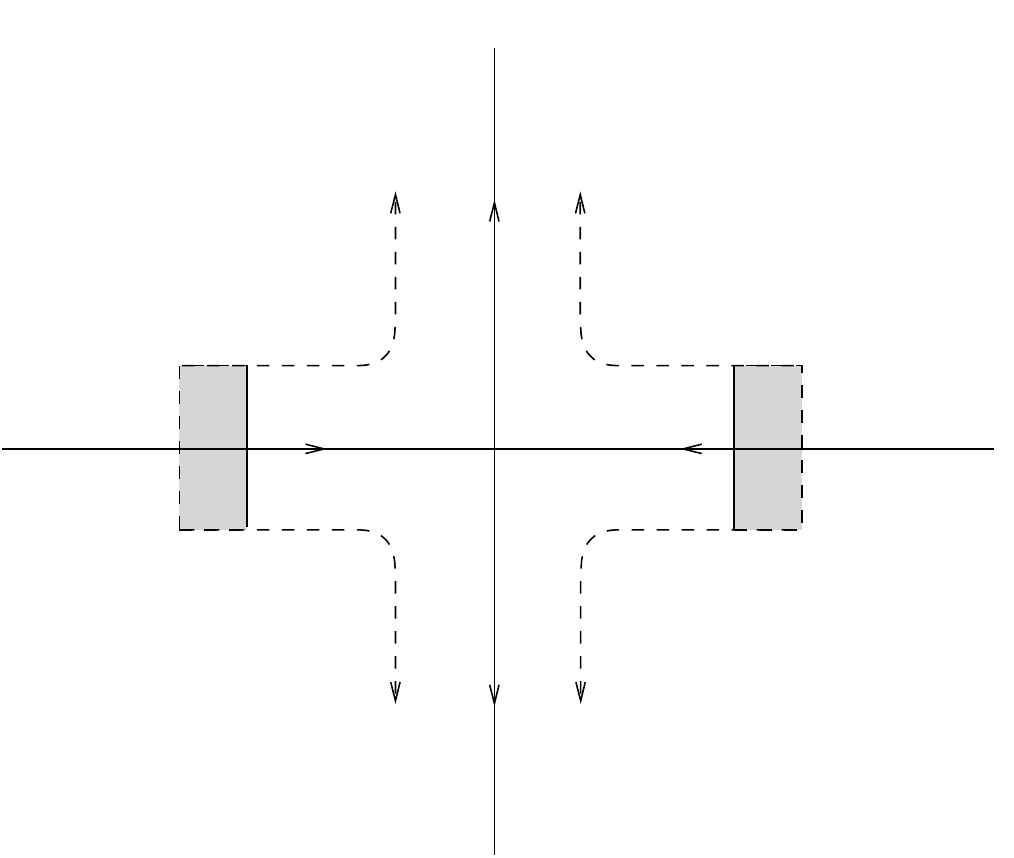}%
\end{picture}%
\setlength{\unitlength}{3947sp}%
\begingroup\makeatletter\ifx\SetFigFont\undefined%
\gdef\SetFigFont#1#2#3#4#5{%
  \reset@font\fontsize{#1}{#2pt}%
  \fontfamily{#3}\fontseries{#4}\fontshape{#5}%
  \selectfont}%
\fi\endgroup%
\begin{picture}(4864,3716)(2102,-6973)
\put(6426,-5236){\makebox(0,0)[lb]{\smash{{\SetFigFont{10}{12}{\rmdefault}{\mddefault}{\updefault}{\color[rgb]{0,0,0}$D\Ver_w^{-}=D_{w}^{p+1}$}%
}}}}
\put(4551,-3499){\makebox(0,0)[lb]{\smash{{\SetFigFont{10}{12}{\rmdefault}{\mddefault}{\updefault}{\color[rgb]{0,0,0}$D\Ver_w^{+}=D_{w}^{q+1}$}%
}}}}
\end{picture}%
\caption{The shaded region denotes the region of the fiber
  $D_w\Ver(F)$ on which the induced metric is defined.}
\label{expdisk}
\end{figure}
Now, via the fiberwise exponential map for $\mathfrak m_{\pi(w)}$, for each
$w\in \Sigma$ we can specify a neighborhood $U_w\subset
\pi^{-1}(\pi(w))$ containing the point $w$ and of the type described in
the proof of Theorem \ref{GLcob}. In Figures \ref{perturb} and
\ref{expdisk}, the image of this region under the inverse exponential
map, before and after the above adjustment of $F$, is shown. For each
$w\in \Sigma$, replace the fiber $D_{w}\Ver_{\Sigma}$ with the image
under the inverse exponential map of $U_w$. Abusing notation we will
retain the name $D\Ver_{\Sigma}$ for this bundle, the fibers of which
should be thought of as the cross-shaped region described in
Fig.~\ref{expdisk}.

The structure group of this bundle is still $O(p+1)\times O(q+1)$. The
metric induced by $\bar{g}_{c-\epsilon_c}$ is defined on a subbundle
with fibers diffeomorphic to $S^{p}\times D^{q+1}\times I$, see
Figure~\ref{expdisk}. On each fiber we now perform the construction
from Theorem \ref{GLcob}. The fact that we adjusted $F$ to make the
trajectories standard on the fiber disk guarantees consistency of the
construction. On each fiber there is a splitting into positive and
negative eigenspaces over which we will perform our construction. We
must however, choose a pair of orthonormal bases for the negative and
positive eigenspaces of that fiber in order to appropriately identify
the fiber with Euclidean space. In order to guarantee consistency we
must ensure that our construction is independent of these choices. But
this follows from Lemma \ref{symmetry}.

Extending the metric fiberwise in the manner of Theorem \ref{GLcob}
and pulling back via the exponential map, gives a smooth family of
fiber metrics, which, with respect to some integrable distribution $H$
and the base metric $\m_B$, combine to the desired submersion metric
on $F^{-1}(B\times[0,c+\delta_c])$.
\end{proof}
  \begin{remark}
    With some little extra care it should be possible to remove the condition
    that $B$ is simply connected in Theorem
    \ref{theo:construct_metrics-1a}. However, we are only interested in the case
    $B=S^n$ with $n>1$ so that, for our purpose, we can stick to the simpler version as
    stated. 
  \end{remark}

\section{Metrics of positive scalar curvature on Hatcher's examples}
\label{sec:Hatchers_examples}
The work of Goette \cite[Section 5.b]{Goette} shows that Hatcher's
examples can be given the structure which is described in Definition
\ref{folds}. The construction of the Hatcher bundles $D^n \to E \to
S^k$ is explained in some detail in \cite{Goette} and will not be
repeated here. Most important for our discussion is the fact that each
of these bundles comes with an admissible family $F$ of Morse
functions as indicated in the following commutative diagram:
\begin{equation}\label{E0}
\begin{diagram}
\setlength{\dgARROWLENGTH}{1.6em}
\node{D^n_z}
      \arrow[2]{e,t}{i_z}
\node[2]{E}
      \arrow{s,l}{\pi}
      \arrow[2]{e,t}{F=(\phi,f)}
\node[2]{S^k\times [0,1/2]}
      \arrow{wsw,b}{p_1}
\\
\node[3]{S^k}
\end{diagram}
\end{equation}
We follow the description given in \cite{Goette}. Each
$f_z:=f|_{E_z}\colon D^n_z \to [0,1/2]$ has three critical points
$p_z^{(0)}$, $p_z^{(1)}$ and $p_z^{(2)}$.  In particular, the points
$p_z^{(0)}$ form a unique fiberwise minimum of the Morse functions
$f_z$ with value $0$, and $F^{-1}(S^k\times \{0\})$ has a neighborhood
$F^{-1}(S^k\times [0,1/8])$ which (as a smooth bundle) is
diffeomorphic to $D^n \times S^k$. Near the value $1/2$, the inverse
image $F^{-1}(S^k\times \{1/2\})$ has a neighborhood diffeomorphic to
$(S^{n-1}\times I) \times S^k$.  We now consider the upside-down copy
of the bundle \eqref{E0}:
\begin{equation}\label{E*}
\begin{diagram}
\setlength{\dgARROWLENGTH}{1.6em}
\node{D^n_z}
      \arrow[2]{e,t}{i_z}
\node[2]{E^*}
      \arrow{s,l}{\pi}
      \arrow[2]{e,t}{F^*}
\node[2]{S^k\times [1/2,1]}
      \arrow{wsw,b}{p_1}
\\
\node[3]{S^k}
\end{diagram}
\end{equation}
Here $E^*:=E$ and
$F^*(e):= (\pi(e),1-f_{\pi(e)}(e))$, i.e. $f^*_z=1-f_z$, where we write
$F=(\phi,f)$. It follows 
that each $f_z^*\colon D^n \to [1/2,1]$ has three
critical points $p_z^{(*0)}$, $p_z^{(*1)}$ and $p_z^{(*2)}$.  In
particular, the points $p_z^{(*0)}$ form a unique fiberwise maximum of
the Morse functions $f_z^*$ with value $1$, and $(F^*)^{-1}(S^k\times
\{0\})$ has a neighborhood $(F^*){-1}(S^k\times [7/8,1])$ 
which (as a smooth bundle) is again diffeomorphic
to $D^n \times S^k$. Near the value $1/2$, the inverse image
$(F^*)^{-1}(S^k\times \{1/2\})$ again has a neighborhood  diffeomorphic
to $(S^{n-1}\times I) \times S^k$. 

By cutting out the neighborhood $D^n\times S^k$ of the
fiberwise minima of $F$, we obtain a smooth bundle
\begin{equation}\label{E1}
\begin{diagram}
\setlength{\dgARROWLENGTH}{1.6em}
\node{(S^{n-1}\times [1/8,1/2])_z}
      \arrow[2]{e,t}{i_z}
\node[2]{E_1}
      \arrow{s,l}{\pi}
      \arrow[2]{e,t}{F_1}
\node[2]{S^k\times [1/8,1/2]}
      \arrow{wsw,b}{p_1}
\\
\node[3]{S^k}
\end{diagram}
\end{equation}
where $E_1:=E\setminus F^{-1}(S^k\times [0,1/8))$,  $F_1:=F|_{E_1}$,
and the spheres $S^{n-1}$ in the product
$$
S^{n-1}\times S^k = F^{-1}(S^k\times \{1/8\})
$$
are given the standard metric $g_0$ of fixed (but arbitrary) radius $b$
independent 
of $z\in S^k$.  The bundle (\ref{E1}) satisfies all the assumptions of
Theorem \ref{theo:construct_metrics-1a}, and we obtain psc-metrics
$\overline{g}_z$ on each fiber $(S^{n-1}\times [1/8,1/2])_z$ with a
product-metric near the boundary. In particular, this gives a family
of metrics $(g_1)_z$ on the spheres $S^{n-1}_z\times \{1/2\}$. 

Now we apply the same construction to the upside-down copy $E^*$ to
obtain a smooth bundle $E^*_1$ with
fibers 
$(S^{n-1}\times [1/2,7/8])_z$. To make sure that the metrics match, we
set $\overline{g}_z^*:=\overline{g}_z$, i.e.~we use the same metric
upside-down.

Because our construction provides metrics which are products near the
boundary, we can glue together the bundles $E_1$ and $E_1^*$ to form a
bundle $\widetilde E \to S^k $ with fiber $(S^{n-1}\times
[1/8,7/8])_z$ together with a smooth family of psc-metrics.  We notice
that the restriction of the bundle $\widetilde E$ to the boundaries
$$
(S^{n-1}\times \{1/8,7/8\})_z=\p (S^{n-1}\times [1/8,7/8])_z
$$
is trivial by construction, and the spheres $S^{n-1}\times
\{1/8,7/8\}$ are given the standard metric independent of the fiber.
Thus we can glue the fiberwise caps $(D^{n}_0\sqcup D_1^n)_z$
to the bundle $\widetilde E \to S^k$ by identifying
$$
\begin{array}{l}
(D^{n}_0)_z  \supset S^{n-1}_z = (S^{n-1}  \times \{1/8\})_z \ ,
\\
\\
(D^{n}_1)_z  \supset S^{n-1}_z = (S^{n-1}  \times \{7/8\})_z \ .
\end{array}
$$
Then we define the torpedo metrics $g_{\tor}(r)$ on the disks
$(D^{n}_0)_z$ and $(D^{n}_0)_z$ such that they match the chosen
standard metric of radius $b$ on the boundary spheres.  We denote the
resulting metric on the fiber sphere $S^n_z$ by $\overline{g}_z$.  Let
$\overline{E} \to S^k$ be the resulting fiber bundle with fiber $S^n$.

Let us investigate what we have achieved: for each $z\in S^k$ we get a
psc-metric on the fiber $S^n_z$ over $z$. This is a manifold
diffeomorphic to $S^n$, but not with a given diffeomorphism.  Hence
this metric defines a point in the moduli space of pcs-metrics on
$S^n$. Finally, there is a base point $z_0 \in S^n$ together with a
fixed neighborhood on which all these diffeomorphisms
restrict to the identity. This implies that in fact we get an element
in $\pi_k{\mathcal M}_{x_0}^+(S^n)$.  The map
$$
\iota: {\mathcal M}_{x_0}^+(S^n) \longrightarrow {\mathcal M}_{x_0}(S^n)  = 
B\Diff_{x_0}(S^n)
$$ 
forgets the fiberwise Riemannian metrics and just remembers the
structure of $\overline E$ as a smooth bundle.  Because for odd $n$
the generators of $\pi_k \mathcal{M}_{x_0}^+(S^n) \otimes \Q$ can in a
stable range be represented by classifying maps of Hatcher bundles
\cite{Bok, Igusa-book, Goette} we have proved our first main result,
Theorem \ref{thm:posit}.

To prove our second main result, Theorem
\ref{thm:posit_general_manifold}, for a general manifold $M$, we use
the above fiber bundles to form non-trivial bundles by taking a
fiberwise connected sum $M\#S^n$.

Let $M$ be a smooth manifold with a base point $x_0$. 
We assume that $M$ is equipped with a psc-metric $h$. We fix a
disk $D_0^n\subset M$ of small radius centered at $x_0$. We may always
deform the metric $h$ near $x_0$ such that its restriction on $D_0^n$
is a torpedo metric $g_{\tor}(r)$ (use e.g.~Lemma \ref{lem:isotopy}
and thinking of the given disk as one half of a tubular neighborhood
of an embedded $S^0$). Thus we will assume that the metric $h$ already
has this property. 

On the other hand, we consider the bundle $\overline{E} 
\to S^k$ with with psc-metrics $\overline{g_z}$ on the fibers 
$S^n_z$, $z\in S^k$, as constructed before.  We notice that the metrics $\overline{g_z}$
are chosen in such way that their restrictions to the disks
$(D^{n}_0)_z$ and $(D^{n}_1)_z$ are torpedo metrics (with chosen
parameter). Let
$$
\widetilde{D^n}_z = S^n_z\setminus (D^{n}_1)_z .
$$
This is a disk together with the metric $\widetilde{g}_z = \overline{g}_z|_{\widetilde{D^n}_z}$ which is a product-metric $g_0+dt^2$ near the
boundary $S^{n-1}_z\subset \widetilde{D^n}_z$. Now for each $z\in S^k$ we
define the Riemannian manifold
$$
M_z = M\# (S^n)_z= (M \setminus D_0) \cup_{S^{n-1}_z} \widetilde{D^n}_z
$$
equipped with the metric $\widetilde{h}_z$ so that
$$
\widetilde{h}_z|_{M \setminus D_0}= h|_{M \setminus D_0}, \ \ \
\widetilde{h}_z|_{D^n_z}= \widetilde{g}_z.
$$
This defines a smooth fiber bundle 
$$
\widetilde E =  (M \times S^k)\# E \longrightarrow S^k,
$$
where $(M \times S^k)\# E$ is the total space of the fiber-wise
connected sum as we just described. It follows from the additivity
property \cite[Section 3.1]{Igusa-new} that the higher
Franz-Reidemeister torsion of the fiber bundle $\widetilde E \to S^k$
is a non-zero class in $H^k(S^k;\Q)$. This implies that the
classifying map
\[
    S^k \to B \Diff_{x_0}(M) = \mathcal{M}_{x_0}(M) 
\]
of this bundle defines a non-zero element in
$\pi_k(\mathcal{M}_{x_0}(M); [\widetilde h])$.  Since we have
constructed psc-metrics on the fibers $M_z$, this non-zero element can
be lifted to $\pi_k({\mathcal M}_{x_0}^+(M),[\widetilde h])$.  This
finishes the proof of Theorem \ref{thm:posit_general_manifold}.

\section{Homotopy type of the usual psc-moduli space} 
\label{sec:asymmetric_manifold}
In this section, we show that for a suitable choice of $M$ as
  in Theorem \ref{thm:posit_general_manifold}, the map
  $\mathcal{M}_{x_0}^+(M)\to \mathcal{M}^+(M)$ is non-trivial on $\pi_k$.

For a closed smooth manifold $M$ let  $A_H(M)$ be the image of 
the canonical map $\Diff(M) \to \Aut(H_*(M;\Q))$.  
\begin{lemma} \label{manifold} For any $N \geq 0$ there is is a closed
  smooth orientable manifold $M$ of dimension $n$ with the following
  properties.
\begin{enumerate}
   \item
  $n$ is odd and $n \geq N$. 
    \item
 $M$ carries a psc-metric. 
   \item
 Each  $S^1$-action on $M$ is trivial. 
   \item
 $A_H(M)$ is finite.
   \item
Each diffeomorphism of $M$ is orientation preserving. 
 \end{enumerate}
\end{lemma}
Before we explain the construction of $M$, we show how Theorem
\ref{thm:usual_moduli} follows.

Let $d > 0$ be given and choose $N$ so that Theorem \ref{thm:posit}
holds for all $n \geq N$ and all $k = 4q \leq d$.  For $k = 4q \leq d$
we consider the fibration
\[
     M \to  E \to S^{k} 
\]
constructed at the end of Section \ref{sec:Hatchers_examples}. 
By construction this fibration is classified by a map $f : S^{k} \to
BG$, where 
$
G := \Torr(M) \cap \Diff_{x_0}(M).
$  
Because the higher Franz-Reidemeister torsion of this bundle is a
non-zero element in $H^{k}(S^{k} ; \Q)$, the fundamental class of
$S^k$ is mapped to a non-zero element in $H_{k}(BG;\Q)$ and then
further to a non-zero element $c \in H_k(B\Torr(M); \Q)$.

Let $\phi : M \to M$ be a diffeomorphism. Then $\phi$ is orientation
preserving by assumption.  Because $\Torr(M)$ is normal in $\Diff(M)$,
the map $\phi$ induces a map $\overline{\phi} : B \Torr(M) \to B
\Torr(M)$, where we think of $B\Torr(M)$ as $E \Diff(M) / \Torr(M)$.
\begin{lemma} $\overline{\phi}_*(c) = c$. 
\end{lemma}
\begin{proof} The map $\overline{\phi}$ is induced my the group
  homomorphism $\Torr(M) \to \Torr(M)$ given by conjugation with
  $\phi$. By construction, $E$ is classified by a map $S^k \to B
  \Diff_{x_0}(M, M -D)$ where $D \subset M$ is a small embedded disc
  around the base point $x_0 \in M$. Note that $\Diff_{x_0}(M, M-D)$
  can be regarded as a subgroup of $\Torr(M)$. The map $\phi$ is
  isotopic to a diffeomorphism fixing $D$.  Conjugation by this
  element induces the identity homomorphism on the subgroup
  $\Diff_{x_0}(M, M- D) \subset \Torr(M)$.
\end{proof}
We conclude that the finite group $A_H(M) = \Diff(M) / \Torr(M)$ acts
freely on the space $E \Diff(M) / \Torr(M) = B \Torr(M)$ with
  quotient $E\Diff(M)/\Diff(M)=B\Diff(M)$ and fixes $c
\in H_*(B \Torr(M) ; \Q)$. A transfer argument implies that $c$ is
mapped to a nonzero class in $H_*(B \Diff(M) ; \Q)$ under the
canonical map $B\Torr(M) \to B\Diff(M)$.

Theorem \ref{thm:usual_moduli} now follows from the observation that
this class lies in the image of the Hurewicz map, from the
commutativity of the diagram
\begin{small}
\[
\xymatrix{
  \Riem^+(M) / G  \ar[r] \ar[d] & \Riem(M) / G  \ar[d]  & \Riem(M)  \!\times_G E \Diff(M)\! = \!BG \ar[d] \ar@{=}[l] \\
  \Riem^+(M) / \Torr(M)\ \ar[r] \ar[d] & \Riem(M) / \Torr(M) \ar[d]   & \Riem(M) \!\times_{\Torr(M)} \!E \Diff(M)\! = \!B \Torr(M) \ar[d] \ar[l]\\
  \Riem^+(M) / \Diff(M) \ar[r] & \Riem(M) / \Diff(M) & \Riem(M)
  \!\times_{\Diff(M)} \!E \Diff(M)\! = \!B \Diff(M)
  \ar[l]_{\hspace*{-1.8cm}\mu} }
\]
\end{small}

\noindent
and from the following lemma. 

\begin{lemma} Assume that $\gamma \in \pi_k(B \Diff(M))$ is
  not in the kernel of the Hurewicz map 
$$
\pi_k(B \Diff(M))\to H_k(B \Diff(M) ; \Q).
$$
Then the  canonical map 
\[
  \mu : B \Diff(M) = \Riem(M) \times_{\Diff(M)}  E \Diff(M)  \to \Riem(M) / \Diff(M) \, , 
\] 
sends $\gamma$ to a non-zero element in $\pi_k(\Riem(M) / \Diff(M))$. 
\end{lemma}
\begin{proof} For $[g] \in \Riem(M) / \Diff(M)$ the preimage
  $\mu^{-1}([g]) = (g \cdot \Diff(M)) \times_{\Diff(M)} E \Diff(M)$ is
  homeomorphic to $B(\Diff(M)_g)$, where $\Diff(M)_g$ is the isotropy
  group of $g \in \Riem(M)$. Furthermore, by the existence of a local
  slice through $g$ for the action of $\Diff(M)$ on $\Riem(M)$, which
  can be assumed to be $\Diff(M)_g$-linear, see for example
  \cite[Section II.13.]{Bourg}, each neighborhood of $[g] \in
  \Riem(M) / \Diff(M)$ contains an open neighborhood $U$ so that
  $\mu^{-1}(U)$ retracts to $\mu^{-1}([g])$. In particular, the Leray
  sheaf $\mathcal{H}^*(\mu)$ for $\mu$, cf. \cite[IV.4]{Bredon}, is
  constant and equal to $\Q$ in degree $0$ and equal to $0$ in all
  other degrees. Here we use the Myers-Steenrod theorem \cite{Myers}
  which says that $\Diff(M)_g$ is a compact Lie group and hence finite
  as $S^1$ can act only trivially on $M$. This implies that the
  reduced sheaf theoretic cohomology
  $\widetilde{H}_{sh}^*(B\Diff(M)_g;\Q) = 0$ for all $g \in \Riem(M)$
  by the usual transfer argument \cite[II.19.]{Bredon} for sheaf
  theoretic cohomology.

  We conclude that the cohomological Leray spectral sequence (see
  e.g. \cite[IV.6]{Bredon})
\[
   E_2^{p,q} = H_{sh}^p(\Riem(M)/\Diff(M) ; 
\mathcal{H}^q(\mu)) \Rightarrow H_{sh}^{p+q}(B \Diff(M) ; \Q) 
\]
collapses at the $E_2$-level. From this it follows that the map $\mu$
induces an isomorphism in sheaf theoretic cohomology with rational
coefficients.

In order to derive the statement of the lemma, note that up to
homotopy equivalence the space $B \Diff(M)$ can be assumed to be a
paracompact Fr\'echet manifold \cite[Section 44.21]{Kr-Mich}, in
particular to be locally contractible.  This and the homotopy
invariance of sheaf theoretic cohomology \cite[Theorem
II.11.12]{Bredon} imply by \cite[Theorem III.1.1.]{Bredon} that there
is a canonical isomorphism
\[
  H_{sh}^*(B \Diff(M) ; \Q) \cong H_{sing}^*(B \Diff(M) ; \Q) 
\]
of sheaf theoretic and singular cohomology. 

Let $\gamma$ be represented by a map $S^k \to B \Diff(M)$ and consider
the composition
\[
   S^k \to B \Diff(M) \to \Riem(M) / \Diff(M) \, . 
\]
We have shown above that there is a class in $H_{sh}^k(\Riem(M)/
\Diff(M) ; \Q)$ whose pull-back under this composition evaluates
non-zero on the singular fundamental class of $S^k$ (after identifying
$H_{sh}^k(S^k;\Q) = H_{sing}^k(S^k;\Q)$). This implies that this
composition cannot be homotopic to a constant map.
\end{proof} 
It remains to construct the manifold $M$ in Lemma \ref{manifold}.

Let $n \geq 3$ be a natural number.  According to Mostow rigidity the
isometry group of a closed hyperbolic $n$-manifold $M$ is isomorphic
to the outer automorphism group $\Out(\pi_1(M))$. In \cite[Theorem
1.1.]{BL} a closed hyperbolic $n$-manifold $M^n$ with trivial isometry
group is constructed. In the notation of {\it loc. cit.}, $M^n$ is
defined as a quotient $\mathbb{H}^n / B$ of hyperbolic $n$-space by a
discrete subgroup of $\Isom(\mathbb{H}^n)$ which, according to Section
2.3. and Remark 6.3. in {\it loc. cit.}, can be assumed to consist
only of orientation preserving isometries of $\mathbb{H}^n$. In
particular, we can assume that $M^n$ is orientable. Summarizing, we
have
\begin{lemma} For each $n \geq 3$, there is an orientable closed
  hyperbolic (hence aspherical) $n$-manifold $B^n$ so that
  $\Out(\pi_1(B^n)) = 1$.
\end{lemma}
Next, let $k \geq 2$ be a natural number. We construct an orientable
$4k$-dimensional manifold $N$ as follows.

Recall the Moore space $M(\mathbb{Z}/2,2) = S^2 \cup_{\phi} D^3$ where
$\phi: \partial D^3 \to S^2$ is of degree $2$.  Its reduced integral
homology is concentrated in degree $2$ and isomorphic to $\Z/2$. Let
$S^2 \to B \SO(3k)$ represent a generator of $\pi_2(B \SO(3k)) =
\Z/2$. This map can be extended to a map $M (\mathbb{Z}/2,2) \to B
\SO(3k)$ which then induces an isomorphism $H^2(B\SO(3k);\Z/2) \cong
H^2(M(\mathbb{Z}/2,2); \Z/2)$ of groups that are isomorphic to
$\Z/2$. By pulling back the universal bundle over $B \SO(3k)$ we
obtain a Euclidean vector bundle $X \to M(\Z/2, 2)$ of rank $3k$ which
is orientable, but not spin. At this point we note that the generator
of $H^2(B \SO(3k);\Z/2)$ is the second Stiefel-Whitney class of the
universal bundle over $B \SO(3k)$.

In this discussion we can replace $M(\mathbb{Z}/2,2)$ by a homotopy
equivalent finite $3$-dimensional simplicial complex, which we denote
by the same symbol.  If $k$ is chosen large enough then $M(\Z/2,2)$
can be embedded as subcomplex in $\R^{k+1}$. We
consider a regular neighborhood $R \subset \R^{k+1}$ of this
subcomplex. This is an compact oriented submanifold of $\R^{k+1}$ with
boundary which contains $M(\Z/2, 2)$ as a deformation retract.  By
definition $\partial R$ is an oriented closed smooth manifold of
dimension $k$. Furthermore, because $R$ has the rational homology of a
point, Poincar\'e duality and the long exact homology sequence for the
pair $(R,\partial R)$ show that $\partial R$ is a rational homology
sphere.  Let $E \to \partial R $ be the restriction of the pull back
over $R$ of the vector bundle $X \to M(\Z/2 ,2)$. If $k$ is chosen
large enough, then $H^{2}(R;\Z/2) \to H^2(\partial R;\Z/2)$ is an
isomorphism and hence $E$ is not spin.

Let $DE$ be the disc bundle of $E$ and let $P$ be the oriented double
of $DE$. The manifold $P$ is the total space of an oriented $S^{3k}$
bundle over $\partial R$ with vanishing Euler class (the latter for
dimension reasons).  Hence the rational homology of $P$ is
concentrated in degrees $0$, $k$, $3k$ and $4k$ and isomorphic to $\Q$
in these degrees.  Furthermore, the manifold $P$ is orientable, but
not spin. The latter holds, because the tangent bundle of $DE$
restricted to $\partial R$ splits as a direct sum $T(\partial R)
\oplus E$ and the bundle $T(\partial R)$ is stably trivial, since it
becomes trivial after adding a trivial real bundle of
rank $1$.

Because $P$ is simply connected by construction, 
the Hurewicz theorem modulo the Serre class of finite abelian groups shows
that $P$ has finite homotopy groups up to degree $k-1$.

If we additionally assume that $k$ is odd, then the only possibly
non-zero Pontrijagin class of $P$ lives in degree $4k$, hence the
$\hat{A}$-genus of $P$ is a multiple of the signature of $P$ and thus
equal to $0$.

There is a $4k$-dimensional oriented closed smooth manifold $Q$, given
by a Milnor $E_8$-plumbing construction \cite{KMil}, which is
$(2k-1)$-connected and whose intersection form on $H^{2k}(Q;\Z)$ is a
direct sum of copies of the positive definite lattice $E_8$, hence
itself a positive definite lattice. In particular, the
  signature of $M$ is non-zero. The first non-zero Pontrijagin class
is $p_k(Q) \in H^{4k}(Q; \Z)$, which is non-zero by
  the signature theorem. In particular the $\hat{A}$-genus of $Q$ is
nonzero.

For later use we recall that positive definite lattices have finite
automorphism groups: Given such a lattice $E$ choose a bounded ball
$D$ around $0$ which contains a set of generators. Because $E$ is
finitely generated and positive definite, $D$ is finite. Now observe
that each automorphism $E$ permutes the points in $D$ and is uniquely
determined by this permutation.

We finally define the oriented manifold $N^{4k} := P \sharp Q$ as the
connected sum of $P$ and $Q$.
\begin{lemma} For each odd $n > 0$ and each (sufficiently large and
  odd) $k > n$, the manifold $M := B^{n} \times N^{4k} $ has all the
  properties described in Proposition \ref{manifold}.
\end{lemma}
\begin{proof} The dimension of $M$ is odd and can be chosen arbitrarily large. 

  The manifold $N$ is simply connected, of dimension at least $5$ (if
  $k$ is large enough) and not spin. It therefore carries a metric of
  positive scalar curvature \cite{GL1} and the same is then true for
  the product $B^{n} \times N^{4k}$.

  Because $B$ is aspherical and $N$ is simply connected, we can regard
  the projection $p_1: M = B \times N \to B$ onto the first factor as
  the classifying map of the universal cover of $M$.  By construction,
  the manifold $M$ has finite $\pi_2$ and $\pi_4$ and the higher
  $\hat{A}$-genus $\langle \mathcal{A}(M) \cup \phi^*(c) , [M]
  \rangle$ associated to the fundamental class $c \in H^{n}(B;\Q) =
  H^{n}(B\pi_1(M); \Q)$ is nonzero. Because the group $\pi_1(B)$ is
  the fundamental group of a hyperbolic manifold, it is torsion free
  and does not contain $\Z^2$ as a subgroup (the latter by Preissman's
  theorem).  This implies that the image of any homomorphism $\Z \to
  \cent(\pi_1(M))$ is trivial. We can therefore apply \cite[Theorem
  4.1.]{HerHer} to conclude that $M$ does not carry any effective
  $S^1$-action.

  We now show that $A_H \subset \Aut(H_*(M;\Q))$ is finite. Let $f : M
  \to M$ be a diffeomorphism. Up to homotopy we can assume that  $f$ 
  fixes a base point $x_0$ so that we get an
  induced automorphism $f_* : \pi_1(M,x_0) \to
  \pi_1(M,x_0)$ and together with the classifying map $p_1 : M \to B$
  a homotopy commutative diagram
\[
 \xymatrix{
   M  \ar[r] \ar[d]^{f}  & B \ar[d]^{Bf_*} \\
   M  \ar[r]           & B    
   }
\]
Because the automorphism $f_* : \pi_1(M) \to \pi_1(M) $
must be inner by our choice of $B$, the right hand vertical map
induces the identity in rational cohomology. The classifying map $\phi
: M \to B$ being an isomorphism in rational cohomology up to degree $n$
(because $k > n$), we see that $f^*$ preserves the subspace
$p_1^*(H^*(B;\Q)) \subset H^*(M;\Q)$ and acts as the identity on this
subspace.

Up to a homotopy equivalence $\widetilde{M} \to N$, the universal
cover $\pi : \widetilde{M} \to B \times N$ can be identified with the
inclusion $N = \{*\} \times N \hookrightarrow B \times N$ (i.e.~the
corresponding triangle diagram commutes). This holds because $B$ is
aspherical and $N$ is simply connected.  Hence, because the
inclusion $N \hookrightarrow B \times
N$ has a left inverse (take the projection $p_2 : B \times N \to N$),
the map $\pi$ identifies $H^*(\widetilde{M} ; \Q)$ with the subspace
$p_2^*(H^*(N;\Q)) \subset H^*(B \times N;\Q)$.  Since $f$ induces a
map $\widetilde{M} \to \widetilde{M}$ (albeit not a map $N \to N$),
$f^*$ preserves this subspace. The induced map on
$H^*(\widetilde{M};\Z)$ defines an automorphism of the lattice
$H^{2k}(\widetilde{M}; \Z) = H^{2k}(Q; \Z)$. Because this lattice is
positive definite, the map $f$ can induce only finitely many self maps
of $H^{2k}(\widetilde{M} ; \Q)$. The remaining nonzero rational
cohomology of $\widetilde{M}$ is concentrated in degrees $0$, $k$,
$3k$ and $4k$ where it is isomorphic to $\Q$. Hence $f^*$ can act only
by minus or plus the identity on these cohomology groups.

We conclude that $f^*$ preserves the subspaces $p_1^*(H^*(B^n;\Q))$
and $p_2^*(H^*(N;\Q))$ of the vector space $H^*(M;\Q)$ and can only act as
the identity on the first and in finitely many ways on the
second. Because $H^*(M ;\Q)$ is generated as a ring by these
subspaces, $f^*$ is determined by the action on these subspaces. This
shows that $A_H$ is indeed finite.

The preceding argument also shows that the induced action of $f$ on
$H^{4k}(\widetilde{M};\Q)$ must be the identity, since a generator of
this group can be chosen as the $k$-th Pontrijagin class of
$\widetilde{M}$ by the construction of $Q$.  This and the fact that
$f^*$ acts trivially on $H^n(M;\Q) = H^n(B^n;Q)$ (see above) imply
that $f$ must act in an orientation preserving fashion on the manifold
$M$.
\end{proof}
 
\end{document}